\documentclass[reqno,english]{amsart}
\usepackage[utf8]{inputenc}
\usepackage{color}
\usepackage{babel}
\usepackage{enumitem}
\usepackage{amstext}
\usepackage{amsthm}
\usepackage{amssymb}
\usepackage{esint}
\usepackage[pdfusetitle,
 bookmarks=true,bookmarksnumbered=false,bookmarksopen=false,
 breaklinks=false,pdfborder={0 0 0},pdfborderstyle={},backref=false,colorlinks=true]
 {hyperref}
\hypersetup{
 pdfborderstyle=,pdfborderstyle={}}

\makeatletter

\DeclareTextSymbolDefault{\textquotedbl}{T1}

\numberwithin{equation}{section}
\numberwithin{figure}{section}
\theoremstyle{plain}
\newtheorem{thm}{\protect\theoremname}
\theoremstyle{definition}
\newtheorem{defn}[thm]{\protect\definitionname}
\theoremstyle{plain}
\newtheorem{prop}[thm]{\protect\propositionname}
\theoremstyle{remark}
\newtheorem{rem}[thm]{\protect\remarkname}
\theoremstyle{definition}
\newtheorem{example}[thm]{\protect\examplename}
\theoremstyle{plain}
\newtheorem{lem}[thm]{\protect\lemmaname}
\theoremstyle{plain}
\newtheorem{cor}[thm]{\protect\corollaryname}

\usepackage{pdfsync}
\usepackage{graphics}
\usepackage{epstopdf}
\usepackage[all]{xy}
\usepackage{amscd}
\usepackage{enumitem}
\usepackage{mathtools}
\usepackage{bbold}
\usepackage[utf8]{inputenc}
\setlist[enumerate]{leftmargin=*,label=(\roman*),align=left}

\makeatletter
\@namedef{subjclassname@2020}{%
  \textup{2020} Mathematics Subject Classification}
\makeatother

\newcommand{\xyR}[1]{ \makeatletter
\xydef@\xymatrixrowsep@{#1} \makeatother} 
\newcommand{\xyC}[1]{ \makeatletter
\xydef@\xymatrixcolsep@{#1} \makeatother} 
\entrymodifiers={++[ ][F]} 

\newcommand{\ra}{\longrightarrow}

\newcommand{\field}[1]{\mathbb{#1}}
\newcommand{\R}{\field{R}} 
\newcommand{\N}{\field{N}} 

\renewcommand{\phi}{\varphi}
\newcommand{\diff}[1]{\ifmmode\mathchoice{\hbox{\rm d}#1}  
 {\hbox{\rm d}#1}  
 {\scalebox{0.75}{$\hbox{\rm d}#1$}}  
 {\scalebox{0.35}{$\hbox{\rm d}#1$}}  
 \fi} 

\newcommand{\abs}[2][\empty]{\ifx#1\empty\left|#2\right|%
\else#1\vert #2 #1\vert\fi}

\newcommand{\frontRise}[2]{\ifmmode\mathchoice{{\vphantom{#1}}^{\scalebox{0.6}{$#2$}}}  
 {{\vphantom{#1}}^{\scalebox{0.56}{$#2$}}}  
 {{\vphantom{#1}}^{\scalebox{0.47}{$#2$}}}  
 {{\vphantom{#1}}^{\scalebox{0.35}{$#2$}}}\fi} 

\newcommand{\frontRiseDown}[3]{\ifmmode\mathchoice{{\vphantom{#1}}^{\scalebox{0.6}{$#2$}}_{\scalebox{0.6}{$#3$}}}  
 {{\vphantom{#1}}^{\scalebox{0.56}{$#2$}}_{\scalebox{0.56}{$#3$}}}  
 {{\vphantom{#1}}^{\scalebox{0.47}{$#2$}}_{\scalebox{0.47}{$#3$}}}  
 {{\vphantom{#1}}^{\scalebox{0.35}{$#2$}}_{\scalebox{0.35}{$#3$}}}\fi} 

\newcommand{\norm}[2][\empty]{\ifx#1\empty\left\Vert#2\right\Vert%
\else#1\Vert #2 #1\Vert\fi}

\makeatother

\providecommand{\corollaryname}{Corollary}
\providecommand{\definitionname}{Definition}
\providecommand{\examplename}{Example}
\providecommand{\lemmaname}{Lemma}
\providecommand{\propositionname}{Proposition}
\providecommand{\remarkname}{Remark}
\providecommand{\theoremname}{Theorem}

\begin{document}
\title[A new approach to weighted Sobolev spaces]{A new approach to weighted Sobolev spaces}
\author{Djamel eddine Kebiche}
\address{\textsc{Faculty of Mathematics, University of Vienna, Austria, Oskar-Morgenstern-Platz
1, 1090 Wien, Austria}}
\thanks{This research was funded in whole or in part by the Austrian Science
Fund (FWF) 10.55776/P33538\\
For open access purposes, the author has applied a CC BY public copyright
license to any author-accepted manuscript version arising from this
submission. }
\email{\texttt{djameleddine.kebiche@univie.ac.at }}
\begin{abstract}
We present in this paper a new way to define weighted Sobolev spaces
when the weight functions are arbitrary small. This new approach can
replace the old one consisting in modifying the domain by removing
the set of points where at least one of the weight functions is very
small. The basic idea is to replace the distributional derivative
with a new notion of weak derivative. In this way, non-locally integrable
functions can be considered in these spaces. Indeed, assumptions under
which a degenerate elliptic partial differential equation has a unique
non-locally integrable solution are given. Tools like a Poincaré inequality
and a trace operator are developed, and density results of smooth
functions are established. 
\end{abstract}

\maketitle

\section{Introduction }

By a weight function $u:\Omega\ra\R$ we mean a locally integrable
function which is non-negative almost everywhere in $\Omega$. For
$p\in[1,\infty)$, $L^{p}(\Omega,u)$ denotes the space of all functions
$f:\Omega\ra\R$ satisfying $fu^{1/p}\in L^{p}(\Omega)$. Let $m\in\N$
and consider a collection $S:=\{u_{\alpha}\}_{\alpha\in\pi_{m}}$
($\pi_{m}:=\{\alpha\in\N^{d}\mid|\alpha|\leq m$\}) of weight functions
on $\Omega$. The weighted Sobolev space $W^{m,p}(\Omega,S)$ is defined
as the space of all functions $f\in L^{p}(\Omega,u)\cap L_{\mathrm{loc}}^{1}(\Omega)$
such that their distributional derivative are elements of $L^{p}(\Omega,u_{\alpha})\cap L_{\mathrm{loc}}^{1}(\Omega)$.
The space $W^{m,p}(\Omega,S)$ is equipped with the norm 
\[
\forall f\in W^{m,p}(\Omega,S):\,||f||_{W^{m,p}(\Omega,S)}:=\left(\sum_{|\alpha|\leq m}||u_{\alpha}^{1/p}D^{\alpha}f||_{L^{p}}^{p}\right)^{1/p}.
\]
A sufficient condition for the space $W^{m,p}(\Omega,S)$ to be a
Banach space is 
\begin{equation}
\forall\alpha\in\pi_{m}:\,u_{\alpha}^{-1}\in L_{\mathrm{loc}}^{\frac{1}{p-1}}(\Omega)\label{eq:1.1}
\end{equation}
(see Thm.~$1.11$ of \cite{Kuf}). Indeed, the assumption (\ref{eq:1.1})
implies that for all $\alpha\in\pi_{m}$, the space $L^{p}(\Omega,u_{\alpha})$
is continuously embedded into $L_{\mathrm{loc}}^{1}(\Omega)$, and
the functional $W^{m,p}(\Omega,S)\ra\R$, $f\ra\int_{\Omega}f\partial^{\alpha}\phi\diff x$
($\phi\in\mathcal{D}(\Omega)$) is continuous, a crucial property
used in the proof of the completeness of the space $W^{m,p}(\Omega,S)$
(see Remark $1.8$ of \cite{Kuf}). However, when (\ref{eq:1.1})
is violated, the weighted Sobolev space $W^{m,p}(\Omega,S)$ is not
necessary a Banach space (see Example 1.12 of \cite{Kuf}). This is
because, when (\ref{eq:1.1}) is violated, the latter functionals
are, in general, no longer continuous. The existing remedy consists
in replacing the set $\Omega$ with a smaller one so that the property
(\ref{eq:1.1}) holds. More precisely, the new set, which will be
denoted by $\Omega_{S}^{*}$, is defined by $\Omega_{S}^{*}:=\Omega\backslash\cup_{\alpha\in\pi_{m}}P(u_{\alpha})$
where $P(u_{\alpha})$ is the set of points $x\in\Omega$ for which
$u_{\alpha}^{-1}\not\in L^{\frac{1}{p-1}}(\Omega\cap U)$ for all
neighborhood $U$ of $x$. The set $B:=\cup_{\alpha\in\pi_{m}}P(u_{\alpha})$
is closed (and hence $\Omega_{S}^{*}$ is open) and Lebesgue negligible.
Moreover, property (\ref{eq:1.1}) holds when $\Omega$ is replaced
with $\Omega_{S}^{*}$ and hence the weighted Sobolev space $W^{m,p}(\Omega_{S}^{*},S)$
is a Banach space (see Sec.~3 of \cite{Kuf} for more details). However,
the disadvantage of this remedy is a modification of the domain. As
a consequence, the convexity might be lost and the boundary of $\Omega$
is changed. 

The alternative we present in this paper consists in replacing the
notion of distributional derivative with a new notion of weak derivative,
called weak derivative in the sense of $L_{v,\mathrm{loc}}^{1}(\Omega)$
($v$ being a $\mathcal{C}^{m}(\Omega)$ function). It consists in
replacing the space of test functions, which is $\mathcal{C}_{c}^{m}(\Omega_{S}^{*})$
in the case of $W^{m,p}(\Omega_{S}^{*},S)$, with a larger one whose
elements do not necessary vanish near the set $B$ (see Def.~\ref{def:=000020Weak=000020derivative=000020P-alpha=000020}).
When the function $v$ is well chosen, i.e.~$\Omega_{S}^{*}=\Omega_{|v|^{mp+p}}^{*}:=\Omega\backslash P(|v|^{mp+p})$,
the two notions of derivative coincide (see Prop.~\ref{prop:comparison=000020with=000020weak=000020derivative=000020}). 

In Sec.~\ref{sec:Density-of-smooth}, several density results of
smooth functions in weighted Sobolev spaces are proved. In Sec.~\ref{sec:The-trace-operator},
we give sufficient conditions under which (continuous) trace operators
can be defined. A Poincaré inequality is established in Sec.~\ref{sec:Poincar=0000E9-inequality}.
In the last section, we deal with degenerate elliptic linear partial
differential equations, where in particular we give assumptions under
which a solution is not locally integrable. 

\section{\protect\label{sec:The-space-L^p}The weighted Lebesgue space $L_{w}^{p}(\Omega)$ }

Throughout this paper, a weight function $u:\Omega\ra\R$ is a locally
integrable function, which is positive a.e. in $\Omega$. The set
of all the weight functions on $\Omega$ is denoted by $W(\Omega)$. 

We now give a definition of weighted Lebesgue spaces (using notations
slightly different from the classical ones) 
\begin{defn}
\label{def:sobolev=000020spaces=000020}Let $\Omega\subseteq\R^{d}$
be an open set, $p\in[1,\infty)$, and let $w$ be such that $|w|^{p}\in W(\Omega)$
(in the sequel we write simply $w\in E_{p}(\Omega)$). We define the
weighted Lebesgue space $L_{w}^{p}(\Omega)$ as the space of measurable
functions $f$ on $\Omega$ such that $fw\in L^{p}(\Omega)$. This
space will be equipped with the norm $||\cdot||_{L_{w}^{p}}$ defined
by 
\begin{equation}
\forall f\in L_{w}^{p}(\Omega):\,||f||_{L_{w}^{p}(\Omega)}:=||fw||_{L^{p}(\Omega)}.\label{eq:=000020weighted=000020norm}
\end{equation}
The space $L_{w}^{2}(\Omega)$ is equipped with the scalar product
$(\cdot,\cdot)_{L_{w}^{2}}$ defined by 
\begin{equation}
\forall f,g\in L_{w}^{2}(\Omega):\,(f,g)_{L_{w}^{2}}:=(wf,wg)_{L^{2}}\label{eq:=000020scalar=000020product-1}
\end{equation}
where $(\cdot,\cdot)_{L^{2}}$ is the usual scalar product of the
space $L^{2}(\Omega)$. 
\end{defn}

Moreover, we write $f\in L_{w,\mathrm{loc}}^{p}(\Omega)$ if $f\in L_{w}^{p}(K)$
for all $K\Subset\Omega$. 

If we denote by $L^{p}(\Omega,u)$ the weighted Lebesgue space with
the weight function $u$, then we clearly have 
\[
L_{w}^{p}(\Omega)=L^{p}(\Omega,|w|^{p}).
\]
The reason for using $w$ instead of the weight function $|w|^{p}$
is purely for notational reason. 

The following theorem is well known (see e.g.~Thm.~III.~6.6 of
\cite{DuSch}).
\begin{thm}
\label{thm:HBT}Let $\Omega\subseteq\R^{d}$ be an open set, $w\in E_{p}(\Omega)$,
and $p\in[1,\infty)$. Then, the space $L_{w}^{p}(\Omega)$ equipped
with the norm (\ref{eq:=000020weighted=000020norm}) is a Banach space
(Hilbert space when $p=2$). 
\end{thm}

Let $\mathcal{S}$ be the Schwartz space. For a fixed non-null polynomial
$P$ we set 
\[
\mathcal{S}_{P}=\{\psi\in\mathcal{S}:\,\psi=P\phi,\phi\in\mathcal{S}\}
\]
which is a linear subspace of $\mathcal{S}$ equipped with the sequence
$(||\cdot||_{k})$ of semi norms of $\mathcal{S}$. Having denoted
by $\mathcal{S}'_{P}$ the dual space of $\mathcal{S}{}_{P}$, we
have the following proposition. 
\begin{prop}
\label{prop:5}Let $P$ be a non-null polynomial function and let
$p\in[1,\infty)$. Then, any element $f\in L_{P}^{p}(\R^{d})$ defines
an element of $\mathcal{S}'_{P}$ in following way 
\[
\forall\psi\in\mathcal{S}_{P}:\,\langle f,\psi\rangle_{\mathcal{S}'_{P},\mathcal{S}{}_{P}}:=\int_{\R^{d}}f(x)\psi(x)\,\diff x.
\]
\end{prop}

\begin{proof}
Let $f\in L_{P}^{p}(\R^{d})$, $\psi\in\mathcal{S}_{P}$ and pick
$\phi$ so that $\psi=P\phi$. By Hölder's inequality we have 
\[
\left|\langle f,\psi\rangle_{\mathcal{S}'_{P},\mathcal{S}_{P}}\right|\leq\int_{\R^{d}}\left|f(x)P(x)\phi(x)\right|\,\diff x\leq||f||_{L_{P}^{p}}||\phi||_{L^{p'}}\leq C||f||_{L_{P}^{p}}||\phi||_{k}
\]
for some $k\in\N$ and $C>0$ that depend only on $d$ and $p$, where
$p'$ is the conjugate exponent of $p$. By Thm.~1 of \cite{Hor},
the linear continuous multiplication map $M_{P}:\mathcal{S}\ra\mathcal{S}_{P}$
defined by $M_{P}\phi=P\phi$ has a continuous inverse $M_{P}^{-1}:\mathcal{S}_{P}\ra\mathcal{S}$.
It follows that 
\[
\exists C'>0\,\exists k'\in\N\,\forall\psi\in\mathcal{S}_{P}:\,\left|\langle f,\psi\rangle_{\mathcal{S}'_{P},\mathcal{S}_{P}}\right|\leq C'||f||_{L_{P}^{p}}||\psi||_{k'}
\]
which completes the proof. 
\end{proof}

\section{\protect\label{sec:Weighted-weak-derivative}The weak derivative
in the sense of $L_{v,\mathrm{loc}}^{1}(\Omega)$}

Let $w\in E_{p}(\Omega)$. When $w$ is arbitrary small, elements
of $L_{w}^{p}(\Omega)$ are not necessary locally integrable over
$\Omega$, and hence the notion of distributional derivative to elements
of $L_{w}^{p}(\Omega)$ is not well defined in general. The existing
remedy is to replace $\Omega$ with $\Omega_{|w|^{p}}^{*}:=\Omega\backslash P(|w|^{p})$
where $P(|w|^{p})$ is the set of points $x\in\Omega$ for which 
\[
|w|^{-p}\not\in L^{\frac{1}{p-1}}(\Omega\cap U)
\]
for all neighborhood $U$ of $x$. The set $P(|w|^{p})$ is closed
and Lebesgue negligible. Moreover, $L_{w}^{p}(\Omega)\subseteq L_{\text{loc}}^{1}(\Omega_{|w|^{p}}^{*})$
with continuous embedding (see Sec.~3 of \cite{Kuf} for more details).
Hence, the notion of weak derivative in the sense of $L_{\mathrm{loc}}^{1}(\Omega_{|w|^{p}}^{*})$
is well defined on the space $L_{w}^{p}(\Omega)$. However, note that
with this weak derivative, we are allowed to take only test functions
having a compact support in $\Omega_{|w|^{p}}^{*}$. 

The main aim of this section is to present a new notion of weak derivative
where the test functions are not necessary equal to $0$ near the
set $P(|w|^{p})$ as it is the case for the weak derivative in the
sense of $L_{\mathrm{loc}}^{1}(\Omega_{|w|^{p}}^{*})$. 
\begin{defn}
\label{def:=000020Weak=000020derivative=000020P-alpha=000020}Let
$\Omega\subseteq\R^{d}$ be an open set, $\alpha\in\N^{d}$ with $|\alpha|\neq0$,
$p\in[1,\infty)$, $w\in E_{p}(\Omega)$, and $v\in\mathcal{C}^{|\alpha|,*}(\Omega)$
(i.e.~$v\in\mathcal{C}^{|\alpha|}(\Omega)$ and $v\neq0$ a.e. in
$\Omega$) be a map satisfying 
\begin{equation}
\forall K\Subset\Omega\,\exists C>0:\,|v|\leq C|w|\,\,\,\text{a.e. in }K.\label{eq:weak=000020derivative=00002011}
\end{equation}
We say that $f\in L_{w,\text{loc}}^{1}(\Omega)$ is $\alpha$-weakly
differentiable in the sense of $L_{v,\text{loc}}^{1}(\Omega)$ if
there exists an element of $L_{v^{|\alpha|+1},\text{loc}}^{1}(\Omega)$
denoted by $D_{v}^{\alpha}f$ satisfying 
\begin{equation}
\forall\phi\in\mathcal{D}(\Omega):\,\int_{\Omega}f\partial^{\alpha}(v^{|\alpha|+1}\phi)\,\diff x=(-1)^{|\alpha|}\int_{\Omega}v^{|\alpha|+1}\phi D_{v}^{\alpha}f\,\diff x.\label{eq:=000020alpha-Weak=000020=000020derivative}
\end{equation}
\end{defn}

In the sequel, for some given $f\in L_{w,\text{loc}}^{1}(\Omega)$,
we will frequently write $\exists D_{v}^{\alpha}f$ to say that the
$\alpha$-weak derivative of $f$ in the sense of $L_{v,\text{loc}}^{1}(\Omega)$
exists (and hence belongs to $L_{v^{|\alpha|+1},\text{loc}}^{1}(\Omega)$). 
\begin{rem}
~\label{rem:weak=000020alpha=000020derivative=000020} Let $\Omega$,
$\alpha$, $p$, $w$, $v$ be as in the previous definition. 
\begin{enumerate}
\item The weak derivative in the sense of $L_{v,\text{loc}}^{1}(\Omega)$
(when it exists) is unique. The proof is identical to the proof of
the uniqueness of the weak derivative in the sense of $L_{\text{loc}}^{1}(\Omega)$. 
\item \label{enu:weak=000020alpha=000020derivative=000020}Take $\alpha$
with $|\alpha|=1$ and let $f\in L_{w}^{p}(\Omega)$. Assume that
$\exists D_{v}^{\alpha}f\in L_{v^{2}}^{p}(\Omega)$, and that $f\partial^{\alpha}(v^{2})$,
$v^{2}f\in L^{p}(\Omega)$, e.g.~$\Omega$ is bounded and $\nabla v\in L^{\infty}(\Omega)$,
then (\ref{eq:=000020alpha-Weak=000020=000020derivative}) yields
\[
\forall\phi\in\mathcal{D}(\Omega):\,\int_{\Omega}\left(f\partial^{\alpha}(v^{2})+v^{2}D_{v}^{\alpha}f\right)\phi\,\diff x=-\int_{\Omega}v^{2}f\partial^{\alpha}\phi\,\diff x
\]
which implies that $v^{2}f$ is $\alpha$-weakly differentiable in
the sense of $L_{\text{loc}}^{1}(\Omega)$ (or simply $\exists D^{\alpha}(v^{2}f)$)
and $D^{\alpha}(v^{2}f)=f\partial^{\alpha}(v^{2})+v^{2}D_{v}^{\alpha}f\in L^{p}(\Omega)$.
Consequently, if we assume that the above conditions hold for all
$\alpha\in\N^{d}$ with $|\alpha|=1$ we get that $v^{2}f\in W^{1,p}(\Omega)$. 
\item \label{enu:=000020continuity=000020weak=000020derivative=000020}One
can show that for any $\phi\in\mathcal{D}(\Omega)$, there exists
a compactly supported continuous function $h_{\alpha}$ on $\Omega$
such that $\partial^{\alpha}(v^{|\alpha|+1}\phi)=vh_{\alpha}$. By
(\ref{eq:weak=000020derivative=00002011}) we have $|\partial^{\alpha}(v^{|\alpha|+1}\phi)|\leq C|w|h_{\alpha}$
a.e. on $K:=\text{supp}(\phi)$. It follows that 
\[
\left|\int_{\Omega}f\partial^{\alpha}(v^{|\alpha|+1}\phi)\,\diff x\right|\leq C||f||_{L_{w}^{p}(\Omega)}||h_{\alpha}||_{L^{\infty}(K)}.
\]
Hence, the left hand side of (\ref{eq:=000020alpha-Weak=000020=000020derivative})
is always finite. Moreover, for any test function $\phi$, the functional
$L_{w}^{p}(\Omega)\ra\mathbb{C}$, $f\ra\int_{\Omega}f\partial^{\alpha}(v^{|\alpha|+1}\phi)\,\diff x$
is continuous. This property will be used to prove the completeness
of the spaces $W_{V,v}^{m,p}(\Omega)$ defined below. 
\item \label{enu:=000020test=000020functions=000020}By density of $\mathcal{D}(\Omega)$
in $\mathcal{C}_{c}^{|\alpha|}(\Omega)$, we can equally well use
the space $\mathcal{C}_{c}^{|\alpha|}(\Omega)$ instead of the space
$\mathcal{D}(\Omega)$ in (\ref{eq:=000020alpha-Weak=000020=000020derivative}).
This is because using (\ref{eq:weak=000020derivative=00002011}) we
have 
\[
\left|\int_{\Omega}f\partial^{\alpha}(v^{|\alpha|+1}\phi)\,\diff x\right|\leq C||fw||_{L^{1}(K)}\sum_{0\leq\beta\leq\alpha}C_{\beta}||g_{\beta}||_{L^{\infty}(K)}||\partial^{\beta}\phi||_{L^{\infty}(K)},
\]
where $K$ is the support of $\phi$, $g_{\beta}$ is a continuous
function that depends on $v$, $\alpha$ and $\beta$. 
\item In case $w\in\mathcal{C}^{|\alpha|}(\Omega)$ we can simply choose
$v=w$. 
\item The case $w\equiv1$ can be considered since (\ref{eq:weak=000020derivative=00002011})
trivially holds. 
\end{enumerate}
\end{rem}

\begin{example}
Let $f(x)=1/x^{2}\in L_{x^{\text{2}}}^{2}(-1,1)$. Then, $\exists D_{x^{2}}f=-2x^{-3}\in L_{x^{3}}^{2}(-1,1)$. 
\end{example}

Now, we compare the notion of weak derivative in the sense of $L_{v,\mathrm{loc}}^{1}(\Omega)$
with the notion of the weak derivative in the sense of $L_{\mathrm{loc}}^{1}$. 

If $m$, $k\in\N$ are such that $k<m$, then $\pi_{m}^{k}:=\pi_{m}\backslash\pi_{k}$.
More generally, if $\alpha\in\N^{d}$, then $\pi_{\alpha}$ is the
set of multi-indices $\beta\in\N^{d}$ such that $\beta\leq\alpha$
(i.e.~$\beta_{i}\leq\alpha_{i}$ for all $i$). If $\lambda\in\N^{d}$
is such that $\lambda<\alpha$ (i.e.~$\lambda\leq\alpha$ with $\lambda_{j}<\alpha_{j}$
for some $j$), then $\pi_{\alpha}^{\lambda}$ denotes the set $\pi_{\alpha}\backslash\pi_{\lambda}$.

The relation between the week derivative in the sense of $L_{\mathrm{loc}}^{1}$
and the weak derivative in the sense of $L_{v,\mathrm{loc}}^{1}(\Omega)$
is given in the next proposition.\textcolor{blue}{{} }
\begin{prop}
\label{prop:comparison=000020with=000020weak=000020derivative=000020}Let
$\Omega\subseteq\R^{d}$ be an open set, $\alpha\in\N^{d}$ with $|\alpha|\neq0$,
$p\in[1,\infty)$, $w\in E_{p}(\Omega)$, $v\in\mathcal{C}^{|\alpha|,*}(\Omega)$
 be a map satisfying (\ref{eq:weak=000020derivative=00002011}). 
\begin{enumerate}
\item \label{enu:=000020comp1}Let $f\in L_{\mathrm{loc}}^{1}(\Omega)$($=L_{w,\mathrm{loc}}^{1}(\Omega)$
with $w\equiv1$). If $\exists D^{\alpha}f\in L_{\mathrm{loc}}^{1}(\Omega)$
then $\exists D_{v}^{\alpha}f$ and $D^{\alpha}f=D_{v}^{\alpha}f$. 
\item \label{enu:comp2}Let $f\in L_{w,\mathrm{loc}}^{p}(\Omega)$. Then,
$f$ is $\alpha$-weakly differentiable in the sense of $L_{v,\mathrm{loc}}^{1}(\Omega)$
with $D_{v}^{\alpha}f\in L_{v^{|\alpha|+1},\mathrm{loc}}^{p}(\Omega)$
if and only if its $\alpha$-weak derivative in the sense of $L_{\mathrm{loc}}^{1}(\Omega_{|v|^{p|\alpha|+p}}^{*})$,
denoted by $D^{\alpha}f$, exists and belongs to $L_{v^{|\alpha|+1},\mathrm{loc}}^{p}(\Omega)$.
In particular, we have $D_{v}^{\alpha}f=D^{\alpha}f$ a.e. in $\Omega_{|v|^{p|\alpha|+p}}^{*}$. 
\end{enumerate}
\end{prop}

\begin{proof}
\ref{enu:=000020comp1}: Recall that, by density of $\mathcal{D}(\Omega)$
in the space $\mathcal{C}_{c}^{|\alpha|}(\Omega)$, we can replace
the space $\mathcal{D}(\Omega)$ in the definition of the weak derivative
in the sense of $L_{\text{loc}}^{1}(\Omega)$ with the space $\mathcal{C}_{c}^{|\alpha|}(\Omega)$.
Hence, it suffices to use the definition of the weak derivative in
the sense of $L_{\text{loc}}^{1}(\Omega)$ with the test function
$\phi v^{|\alpha|+1}\in\mathcal{C}_{c}^{|\alpha|}(\Omega)$ instead
of $\phi\in\mathcal{D}(\Omega)$.

\ref{enu:comp2}: Assume that $f$ is $\alpha$-weakly differentiable
in the sense of $L_{v,\mathrm{loc}}^{1}(\Omega)$ with $D_{v}^{\alpha}f\in L_{v^{|\alpha|+1},\mathrm{loc}}^{p}(\Omega)$.
Let $\eta\in\mathcal{C}^{\infty}(\R)$ be such that 
\[
\eta\equiv1\,\,\,\text{on }\,\left[-1,1\right]^{c},\,\eta\equiv0\text{ on }\left[-\frac{1}{2},\frac{1}{2}\right],\,0\leq\eta\leq1,
\]
and for all $n\in\N$, set $\chi_{n}:=\eta_{n}\circ v$, where $\eta_{n}(\cdot):=\eta(n\cdot)$.
Setting 
\[
\forall n\in\N:\,M_{n}:=\left\{ x\in\Omega\mid|v(x)|\leq1/n\right\} 
\]
Clearly $\chi_{n}\in\mathcal{C}^{|\alpha|}(\Omega)$, and 
\begin{equation}
\forall n\in\N:\,\chi_{n}\equiv1\,\,\,\text{on }M_{n}^{c},\,\chi_{n}\equiv0\text{ on }M_{2n},\,\text{and }0\leq\chi_{n}\leq1\label{eq:chi_n}
\end{equation}
where $M_{n}^{c}$ is the complementary (within $\Omega$) of $M_{n}$.
For any $\phi\in\mathcal{D}(\Omega_{|v|^{p|\alpha|+p}}^{*})$, we
are allowed to use (\ref{eq:=000020alpha-Weak=000020=000020derivative})
with $\chi_{n}\phi v^{-|\alpha|-1}\in\mathcal{C}_{c}^{|\alpha|}(\Omega_{|v|^{p|\alpha|+p}}^{*})$
(see \ref{enu:=000020test=000020functions=000020} of Rem.~\ref{rem:weak=000020alpha=000020derivative=000020}).
Thus, we obtain 
\[
\sum_{0\leq\beta\leq\alpha}{\alpha \choose \beta}\int_{\Omega}f\partial^{\beta}\chi_{n}\partial^{\alpha-\beta}\phi\,\diff x=(-1)^{|\alpha|}\int_{\Omega}\chi_{n}\phi D_{v}^{\alpha}f\,\diff x.
\]
By the continuous inclusions $L_{w,\mathrm{loc}}^{p}(\Omega)\subseteq L_{v,\mathrm{loc}}^{p}(\Omega)\subseteq L_{v^{|\alpha|+1},\mathrm{loc}}^{p}(\Omega)\subseteq L_{\mathrm{loc}}^{1}(\Omega_{v^{p|\alpha|+p}}^{*})$
(the first inclusion follows from (\ref{eq:weak=000020derivative=00002011}),
and the latter inclusion follows from Thm.~$1.5$ of \cite{Kuf})
$f$, $D_{v}^{\alpha}f\in L_{\mathrm{loc}}^{1}(\Omega_{|v|^{p|\alpha|+p}}^{*})$.
Hence, by the dominated convergence theorem, we have that 
\[
\int_{\Omega}\chi_{n}\phi D_{v}^{\alpha}f\,\diff x.\ra\int_{\Omega}\phi D_{v}^{\alpha}f\,\diff x\,\,\,\text{and}\,\,\,\int_{\Omega}f\chi_{n}\partial^{\alpha}\phi\,\diff x\ra\int_{\Omega}f\partial^{\alpha}\phi\,\diff x.
\]
Left to show that 
\begin{equation}
\forall\beta\in\pi_{\alpha}^{0}:\,\int_{\Omega}f\partial^{\beta}(\chi_{n})\partial^{\alpha-\beta}\phi\,\diff x\ra0.\label{eq:2.5555}
\end{equation}
Fix $\beta\in\pi_{\alpha}^{0}$. Note that $\partial^{\beta}\chi_{n}\to0$
pointwise a.e. in $\Omega$. On the other hand, the Faà di Bruno formula
(\cite{wiki}) gives 
\[
\partial^{\beta}\chi_{n}(x)=\sum_{k=1}^{|\beta|}n^{k}|\{\pi\in\Pi\,\big|\,|\pi|=k\}|\eta^{(k)}(nv(x))\sum_{\pi\in\Pi,|\pi|=k}\prod_{B\in\pi}\frac{\partial^{|B|}v}{\prod_{j\in B}\partial x_{j}}
\]
where 
\begin{itemize}
\item $\pi$ runs through the set $\Pi$ of all partitions of the set $\mathcal{A}$
of non-null indices of $\beta$ with multiplicity e.g.~if $\beta=(2,3,0,1)$
then $\mathcal{A}:=\{1,1,2,2,2,4\}$
\item ``$B\in\pi"$ means the variable $B$ runs through the list of all
of the \textquotedbl blocks\textquotedbl{} of the partition $\pi$,
and 
\item $|A|$ denotes the cardinality of the set $A$ (so that $|\pi|$ is
the number of blocks in the partition $\pi$ and $|B|$ is the size
of the block $B$).
\end{itemize}
By construction, $\partial^{\beta}\chi_{n}(x)=0$ for all $x\in M_{n}^{c}\cup M_{2n}$.
Let now $x\in M_{n}\cap M_{2n}^{c}$. Since $v\in\mathcal{C}^{|\alpha|}(\Omega)$,
we have that 
\[
\exists C>0\,\forall n\in\N\,\forall x\in\text{supp}(\phi)\cap M_{n}\cap M_{2n}^{c}:\,|\partial^{\beta}\chi_{n}(x)|\leq Cn^{|\beta|},
\]
and hence 
\[
\exists C>0\,\forall n\in\N\,\forall x\in\text{supp}(\phi)\cap M_{n}\cap M_{2n}^{c}:\,|v^{|\alpha|}\partial^{\beta}\chi_{n}(x)|\leq Cn^{|\beta|-|\alpha|}\leq C.
\]
It follows that 
\[
\exists C>0\,\forall n\in\N\,:\,\left|f\partial^{\beta}(\chi_{n})\partial^{\alpha-\beta}\phi\right|\leq C\left|v^{-|\alpha|}f\partial^{\alpha-\beta}\phi\right|\,\,\,\text{a.e. in }\Omega.
\]
Since $f\in L_{v,\mathrm{loc}}^{p}(\Omega)$ (by (\ref{eq:weak=000020derivative=00002011}))
and $v^{-|\alpha|-1}\in L_{\mathrm{loc}}^{\frac{p}{p-1}}(\Omega_{|v|^{p|\alpha|+p}}^{*})$
(which follows from the definition of the set $\Omega_{|v|^{p|\alpha|+p}}^{*}$),
Hölder's inequality implies that $v^{-|\alpha|}f\partial^{\alpha-\beta}\phi\in L^{1}(\Omega)$.
Therefore, (\ref{eq:2.5555}) is now a consequence of the dominated
convergence theorem. 

Assume now that $f$ is $\alpha$-weakly differentiable in the sense
of $L_{\mathrm{loc}}^{1}(\Omega_{|v|^{p|\alpha|+p}}^{*})$ with $D^{\alpha}f\in L_{v^{|\alpha|+1},\mathrm{loc}}^{p}(\Omega)$.
Let $\chi_{n}$ be defined as above. Since $\chi_{n}v^{|\alpha|+1}\phi\in\mathcal{C}_{c}^{|\alpha|}(\Omega_{v^{p|\alpha|+p}}^{*})$,
we have 
\begin{multline}
\sum_{0\leq\beta\leq\alpha}{\alpha \choose \beta}\int_{\Omega}f\partial^{\beta}(\chi_{n})\partial^{\alpha-\beta}(v^{|\alpha|+1}\phi)\,\diff x=\\
\int_{\Omega}f\partial^{\alpha}(v^{|\alpha|+1}\chi_{n}\phi)\,\diff x=(-1)^{|\alpha|}\int_{\Omega}v^{|\alpha|+1}\chi_{n}\phi D^{\alpha}f\,\diff x.\label{eq:2.11}
\end{multline}
One can easily show by induction that 
\begin{equation}
\forall\beta\in\pi_{\alpha}\,\exists g_{\beta}\in\mathcal{C}_{c}^{|\beta|}(\Omega):\,\partial^{\alpha-\beta}(v^{|\alpha|+1}\phi)=v^{|\beta|+1}g_{\beta}.\label{eq:induction}
\end{equation}
As shown above, it is easy to see that 
\[
\forall\beta\in\pi_{\alpha}\,\exists C>0\,\forall n\in\N:\,\left|f\partial^{\alpha-\beta}(v^{|\alpha|+1}\phi)\partial^{\beta}\chi_{n}\right|\leq C|fvg_{\beta}|
\]
By the dominated convergence theorem (taking the limit in (\ref{eq:2.11}))
we get 
\[
\int_{\Omega}f\partial^{\alpha}(v^{|\alpha|+1}\phi)\,\diff x=(-1)^{|\alpha|}\int_{\Omega}v^{|\alpha|+1}\phi D^{\alpha}f\,\diff x,
\]
which completes the proof. 
\end{proof}

\section{\protect\label{sec:The-space=000020W=000020m=000020p=000020}The
weighted Sobolev space $W_{V,v}^{m,p}(\Omega)$}

If $V:=\{w_{\alpha}\}_{\alpha\in\pi_{m}}$ is a collection of elements
of $E_{p}(\Omega)$, then the function $w_{0}$ ($0\in\N^{d}$) will
be simply denoted by $w$. 
\begin{defn}
\label{def:generalized=000020weighted=000020sobolev=000020space=000020}Let
$\Omega\subseteq\R^{d}$ be an open set, $p\in[1,\infty)$, $m\in\N$,
and $V:=\{w_{\alpha}\}_{\alpha\in\pi_{m}}$ be a collection of elements
of $E_{p}(\Omega)$, $v\in\mathcal{C}^{m,*}(\Omega)$ be a map satisfying
(\ref{eq:weak=000020derivative=00002011}). We set 
\[
W_{V,v}^{m,p}(\Omega):=\{f\in L_{w}^{p}(\Omega)\mid\forall\alpha\in\pi_{m}^{0}:\,\exists D_{v}^{\alpha}f\in L_{w_{\alpha}}^{p}(\Omega)\}.
\]
This space is equipped with the norm 
\[
\forall f\in W_{V,v}^{m,p}(\Omega):\,||f||_{W_{V,v}^{m,p}}:=\left(\sum_{|\alpha|\leq m}||D_{v}^{\alpha}f||_{L_{w_{\alpha}}^{p}}^{p}\right)^{1/p}.
\]
When $p=2$, the space $W_{V,v}^{m,2}(\Omega)$ will be denoted by
$H_{V,v}^{m}(\Omega)$ and its norm is generated by the scalar product
\[
\forall f,g\in H_{V,v}^{m}(\Omega):\,(f,g)_{H_{V,v}^{m}}:=\sum_{|\alpha|\leq m}(D_{v}^{\alpha}f,D_{v}^{\alpha}g)_{L_{w_{\alpha}}^{2}}.
\]
\textcolor{blue}{}
\end{defn}

We now prove that the space $W_{V,v}^{m,p}(\Omega)$ is Cauchy complete. 
\begin{thm}
\label{thm:W^m,p=000020complete=000020}Let $\Omega\subseteq\R^{d}$
be an open set, $p\in[1,\infty)$, $m\in\N$, and let $V:=\{w_{\alpha}\}_{\alpha\in\pi_{m}}$
be a collection of elements of $E_{p}(\Omega)$, and $v\in\mathcal{C}^{m,*}(\Omega)$.
Assume that 
\begin{equation}
\forall\alpha\in\pi_{m}\,\forall K\Subset\Omega\,\exists C_{K}:\,|v(x)|^{|\alpha|+1}\leq C_{K}|w_{\alpha}(x)|\,\,\,\text{a.e. in }K.\label{eq:=000020w^=00005Calpha<=000020w_=00005Calpha}
\end{equation}
Then, the space $W_{V,v}^{m,p}(\Omega)$ is Cauchy complete. 
\end{thm}

\begin{proof}
Let $(f_{n})$ be a Cauchy sequence of $W_{V,v}^{m,p}(\Omega)$. By
definition, for all $\alpha\in\pi_{m}$, the sequence $(D_{v}^{\alpha}f_{n})$
is a Cauchy sequence of $L_{w_{\alpha}}^{p}(\Omega)$. By completeness
of $L_{w_{\alpha}}^{p}(\Omega)$ (see Thm.~\ref{thm:HBT}) there
exists a unique $f_{\alpha}\in L_{w_{\alpha}}^{p}(\Omega)$ limit
of the sequence $(D_{v}^{\alpha}f_{n})$ in $L_{w_{\alpha}}^{p}(\Omega)$.
Set $f:=f_{0}$. We shall prove that $f$ is weakly differentiable
in the sense of $L_{v,\text{loc}}^{1}(\Omega)$ with $D_{v}^{\alpha}f=f_{\alpha}$
for all $\alpha\in\pi_{m}^{0}$. Indeed, by assumption we have 
\begin{equation}
\forall\phi\in\mathcal{D}(\Omega):\,\int_{\Omega}f_{n}\partial^{\alpha}\left(v^{|\alpha|+1}\phi\right)\,\diff x=(-1)^{|\alpha|}\int_{\Omega}\phi v^{|\alpha|+1}D_{v}^{\alpha}f_{n}\,\diff x.\label{eq:0.4}
\end{equation}
We have $\lim_{n\to\infty}\int_{\Omega}f_{n}\partial^{\alpha}\left(v^{|\alpha|+1}\phi\right)\,\diff x=\int_{\Omega}f\partial^{\alpha}\left(v^{|\alpha|+1}\phi\right)\,\diff x$
by \ref{enu:=000020continuity=000020weak=000020derivative=000020}
of Rem.~\ref{rem:weak=000020alpha=000020derivative=000020}. On the
other hand, (\ref{eq:=000020w^=00005Calpha<=000020w_=00005Calpha})
implies that 
\[
\left|\int_{\Omega}(D_{v}^{\alpha}f_{n})v^{|\alpha|+1}\phi\,\diff x\right|\leq C_{\text{supp}(\phi)}\int_{\Omega}\left|(D_{v}^{\alpha}f_{n})w_{\alpha}\phi\right|\,\diff x\le C_{\text{supp}(\phi)}||D_{v}^{\alpha}f_{n}||_{L_{w_{\alpha}}^{p}}||\phi||_{L^{q}}
\]
where $q$ is the conjugate exponent of $p$, which implies that the
right hand side of (\ref{eq:0.4}) converges to $(-1)^{|\alpha|}\int_{\Omega}\phi v^{|\alpha|+1}f_{\alpha}\,\diff x$
when $n\to\infty$. It follows that 
\[
\forall\phi\in\mathcal{D}(\Omega):\,\int_{\Omega}f\partial^{\alpha}\left(v^{|\alpha|+1}\phi\right)\,\diff x=(-1)^{|\alpha|}\int_{\Omega}\phi v^{|\alpha|+1}f_{\alpha}\,\diff x
\]
which proves that $\exists D_{v}^{\alpha}f=f_{\alpha}\in L_{w_{\alpha}}^{p}(\Omega)\subseteq L_{v^{|\alpha|+1},\text{loc}}^{1}(\Omega)$
(where the last inclusion follows from (\ref{eq:=000020w^=00005Calpha<=000020w_=00005Calpha})),
and hence $f\in W_{V,v}^{m,p}(\Omega)$. Therefore, the sequence $(f_{n})$
converges to $f$ in $W_{V,v}^{m,p}(\Omega)$, and the proof is completed. 
\end{proof}
\begin{rem}
\label{rem:w_=00005Calpha=00003Dw^i} If $w_{\alpha}=|v|^{\sigma}$
for some $\sigma\in(0,|\alpha|+1)$ then (\ref{eq:=000020w^=00005Calpha<=000020w_=00005Calpha})
holds. This is because $|v|^{|\alpha|+1-\sigma}$ is bounded on any
compact subset of $\Omega$. 
\end{rem}

We now compare the new approach of weighted Sobolev spaces presented
in this paper with the classical one. 

Let $V:=\{w_{\alpha}\}_{\alpha\in\pi_{m}}$ be a collection of elements
of $E_{p}(\Omega)$, and let $S:=\{u_{\alpha}\}_{\alpha\in\pi_{m}}$
where $u_{\alpha}=|w_{\alpha}|^{p}\in W(\Omega)$. We recall that
the weighted Sobolev space $W^{m,p}(\Omega,S)$ is defined as the
space of all functions $f\in L^{p}(\Omega,u)\cap L_{\mathrm{loc}}^{1}(\Omega)$
such that their distributional derivative are elements of $L^{p}(\Omega,u_{\alpha})\cap L_{\mathrm{loc}}^{1}(\Omega)$.
The space $W^{m,p}(\Omega,S)$ is equipped with the norm 
\[
\forall f\in W^{m,p}(\Omega,S):\,||f||_{W^{m,p}(\Omega,S)}:=\left(\sum_{|\alpha|\leq m}||u_{\alpha}^{1/p}D^{\alpha}f||_{L^{p}}^{p}\right)^{1/p}.
\]
In order to guarantee the Cauchy completeness of $W^{m,p}(\Omega,S)$
we assume that 
\begin{equation}
\forall\alpha\in\pi_{m}:\,u_{\alpha}^{-1}\in L_{\mathrm{loc}}^{\frac{1}{p-1}}(\Omega)\label{eq:B_p}
\end{equation}
(see Thm.~$1.11$ of \cite{Kuf}). However, when (\ref{eq:B_p})
is violated, the weighted Sobolev space $W^{m,p}(\Omega,S)$ is not
necessarily a Banach space (see Example 1.12 of \cite{Kuf}). Indeed,
the ``bad'' set which may cause the non-completeness of $W^{m,p}(\Omega,S)$
is the set $B:=\cup_{\alpha\in\pi_{m}}P(u_{\alpha})$ where $P(u_{\alpha})$
is the set of points $x\in\Omega$ for which $u_{\alpha}^{-1}\not\in L^{\frac{1}{p-1}}(\Omega\cap U)$
for all neighborhood $U$ of $x$. The existing remedy is to define
the space $W^{m,p}(\Omega,S)$ as the space $W^{m,p}(\Omega_{S}^{*},S)$
where $\Omega_{S}^{*}:=\Omega\backslash B$. It is shown that the
set $B$ is closed (and hence $\Omega_{S}^{*}$ is open) and Lebesgue
negligible. Moreover, assumption (\ref{eq:B_p}) holds if we replace
$\Omega$ with $\Omega_{S}^{*}$. Therefore, the weighted Sobolev
space $W^{m,p}(\Omega,S)$($=W^{m,p}(\Omega_{S}^{*},S)$) is now Cauchy
complete (see Sec.~3 of \cite{Kuf} for more details). However, the
space of test functions used in the definition of the distributional
derivative is, in general, limited to the space $\mathcal{C}_{c}^{m}(\Omega_{S}^{*})$. 

On the other hand, if there exists a $\mathcal{C}^{m,*}(\Omega)$
function $v$ that satisfies (\ref{eq:=000020w^=00005Calpha<=000020w_=00005Calpha}),
then, by assertion \ref{enu:comp2} of Prop.~\ref{prop:comparison=000020with=000020weak=000020derivative=000020},
we have that 
\begin{equation}
W^{m,p}(\Omega_{S}^{*},S)\subseteq W^{m,p}(\Omega_{|v|^{pm+p}}^{*},S)=W_{V,v}^{m,p}(\Omega).\label{eq:2.13}
\end{equation}
In case $\Omega_{S}^{*}=\Omega_{v^{pm+p}}^{*}$ e.g.~$w_{\alpha}:=v^{|\alpha|+1}$,
the inclusion in (\ref{eq:2.13}) becomes an equality. 

In conclusion, we have the following remark 
\begin{rem}
For a given collection $S:=\{|w_{\alpha}|^{p}\}_{\alpha\in\pi_{m}}$
of arbitrarily small weight functions, the weighted Sobolev space
$W^{m,p}(\Omega_{S}^{*},S)$ is a Banach space while the space $W^{m,p}(\Omega,S)$
in general is not. The price we pay is a modification of the domain
and hence the convexity property might be lost and the boundary of
$\Omega$ is modified. If there exists a $\mathcal{C}^{m,*}(\Omega)$
function $v$ satisfying (\ref{eq:=000020w^=00005Calpha<=000020w_=00005Calpha}),
then we can consider the space $W_{V,v}^{m,p}(\Omega)$ (where $V:=\{w_{\alpha}\}_{\alpha\in\pi_{m}}$)
instead of $W^{m,p}(\Omega_{S}^{*},S)$, which is in general larger.
In this case, no modification of the domain is needed and elements
of $W_{V,v}^{m,p}(\Omega)$ are weakly differentiable in the sense
of $L_{\mathrm{loc}}^{1}(\Omega_{|v|^{mp+p}}^{*})$ and in the sense
of $L_{v,\mathrm{loc}}^{1}(\Omega)$ but not necessary in the sense
of $L_{\mathrm{loc}}^{1}(\Omega_{S}^{*})$. If furthermore, the function
$v$ is well chosen so that $\Omega_{S}^{*}=\Omega_{|v|^{pm+p}}^{*}$,
then $W^{m,p}(\Omega_{S}^{*},S)=W_{V,v}^{m,p}(\Omega).$ 
\end{rem}

Let $m\in\N$ and let $s:\pi_{m}\ra\R_{\geq0}$ be a map having the
following properties 
\begin{equation}
\forall\alpha\in\pi_{m}\,\forall\beta\in\pi_{\alpha}:\,s(\alpha)\leq|\alpha|,\,\,\,\text{and}\,\,\,s(\alpha-\beta)+s(\beta)\leq s(\alpha)\label{eq:=000020the=000020map=000020s}
\end{equation}
A typical example of $s$ is the map $|\cdot|$. Let $v\in\mathcal{C}^{m,*}(\Omega)$,
and consider the collection $V=\{w_{\alpha}\}_{\alpha\in\pi_{m}}$
of elements of $E_{p}(\Omega)$ defined by 
\[
\forall\alpha\in\pi_{m}:w_{\alpha}:=|v|^{s(\alpha)+1}.
\]
Then, clearly (\ref{eq:=000020w^=00005Calpha<=000020w_=00005Calpha})
holds since $|\alpha|-s(\alpha)\geq0$. It follows that the space
$W_{s,v}^{m,p}(\Omega):=W_{V,v}^{m,p}(\Omega)$ is a Banach space
(see Thm.~\ref{thm:W^m,p=000020complete=000020}). 
\begin{prop}
\label{prop:continuous=000020embedding=000020W^mp}Let $\Omega\subseteq\R^{d}$
be an open set, $p\in[1,\infty)$, $m\in\N$, $v\in\mathcal{C}^{m,*}(\Omega)$,
$s:\pi_{m}\ra\R_{\geq0}$ be a map satisfying (\ref{eq:=000020the=000020map=000020s}).
Then, the operator $W_{s,v}^{m,p}(\Omega)\ra W_{\mathrm{loc}}^{m,p}(\Omega)$,
$f\ra v^{m+1}f$ is continuous. Moreover 
\begin{equation}
\forall f\in W_{s,v}^{m,p}(\Omega)\,\forall\alpha\in\pi_{m}:\,D^{\alpha}(v^{m+1}f)=\sum_{0\leq\beta\leq\alpha}{\alpha \choose \beta}D_{v}^{\beta}f\partial^{\alpha-\beta}v^{m+1}.\label{eq:Leibniz=000020form}
\end{equation}
In particular, one can replace $W_{\mathrm{loc}}^{m,p}(\Omega)$ with
$W^{m,p}(\Omega)$ in case $v\in\mathcal{C}_{b}^{m}(\Omega)$. 
\end{prop}

\begin{proof}
Let $f\in W_{s,v}^{m,p}(\Omega)$ and $\alpha\in\pi_{m}$. For any
$\sigma\in\pi_{\alpha}$, there exists $g_{\sigma}\in\mathcal{C}^{m-|\sigma|}(\Omega)$
such that $\partial^{\sigma}v^{m+1}=g_{\sigma}v^{m+1-|\sigma|}$,
which implies that 
\begin{equation}
\forall K\Subset\Omega:\,\left\Vert D_{v}^{\beta}f\partial^{\alpha-\beta}v^{m+1}\right\Vert _{L^{p}(K)}\leq C||v^{m-|\alpha|}g_{\alpha-\beta}||_{L^{\infty}(K)}\left\Vert |v|^{s(\beta)+1}D_{v}^{\beta}f\right\Vert _{L^{p}(\Omega)}\label{eq:=000020continuity}
\end{equation}
where we used (\ref{eq:=000020w^=00005Calpha<=000020w_=00005Calpha})
(with $w_{\beta}=|v|^{s(\beta)+1}$) in the last inequality, which
shows that the right hand side of (\ref{eq:Leibniz=000020form}) belongs
to $L_{\mathrm{loc}}^{p}(\Omega)$ for all $\alpha\in\pi_{m}$. In
case $v\in\mathcal{C}_{b}^{m}(\Omega)$, $v^{m-|\alpha|}g_{\alpha-\beta}\in L^{\infty}(\Omega)$
and the constant $C$ does not depend on the compact set $K$. Thus,
(\ref{eq:=000020continuity}) holds with $\Omega$ instead of $K$.
Therefore, the right hand side of (\ref{eq:Leibniz=000020form}) belongs
to $L^{p}(\Omega)$ for all $\alpha\in\pi_{m}$. We show now that
$v^{m+1}f$ is weakly differentiable in the sense of $L_{\mathrm{loc}}^{1}(\Omega)$
and that the Leibniz formula (\ref{eq:Leibniz=000020form}) holds.
Indeed, we will prove this by induction. Suppose first that $\alpha\in\pi_{1}^{0}$.
Choose $\phi\in\mathcal{D}(\Omega)$. We have
\begin{multline*}
\int_{\Omega}fv^{m+1}\partial^{\alpha}\phi\,\diff x=\int_{\Omega}f(\partial^{\alpha}(v^{m+1}\phi)-\phi\partial^{\alpha}v^{m+1})\,\diff x\\
=-\int_{\Omega}(v^{m+1}D_{v}^{\alpha}f+f\partial^{\alpha}v^{m+1})\phi\,\diff x,
\end{multline*}
where we used (\ref{eq:=000020alpha-Weak=000020=000020derivative})
with the test function $\phi v^{m-|\alpha|}\in\mathcal{C}_{c}^{m}(\Omega)$,
which proves (\ref{eq:Leibniz=000020form}) when $\alpha\in\pi_{1}^{0}$.
We assume now that (\ref{eq:Leibniz=000020form}) holds for all $\alpha\in\pi_{l}^{0}$
for some $l<m$. Choose now $\alpha\in\pi_{l+1}^{l}$. Then $\alpha=\beta+\gamma$
for some $|\beta|=l$ and $|\gamma|=1$. Then for $\phi$ as above,
\begin{multline*}
\int_{\Omega}fv^{m+1}\partial^{\alpha}\phi\,\diff x=(-1)^{|\beta|}\int_{\Omega}D^{\beta}(fv^{m+1})\partial^{\gamma}\phi\,\diff x\\
=(-1)^{|\beta|}\sum_{0\leq\sigma\leq\beta}{\beta \choose \sigma}\int_{\Omega}D_{v}^{\sigma}f\partial^{\beta-\sigma}v^{m+1}\partial^{\gamma}\phi\,\diff x
\end{multline*}
where we used the induction assumption. Using the equality 
\[
\partial^{\beta-\sigma}v^{m+1}\partial^{\gamma}\phi=\partial^{\gamma}(\phi\partial^{\beta-\sigma}v^{m+1})-\phi\partial^{\alpha-\sigma}v^{m+1}
\]
 we get 
\begin{multline*}
\int_{\Omega}fv^{m+1}\partial^{\alpha}\phi\,\diff x=(-1)^{|\beta|}\sum_{0\leq\sigma\leq\beta}{\beta \choose \sigma}\int_{\Omega}D_{v}^{\sigma}f\partial^{\gamma}(\phi\partial^{\beta-\sigma}v^{m+1})\,\diff x\\
+(-1)^{|\alpha|}\sum_{0\leq\sigma\leq\beta}{\beta \choose \sigma}\int_{\Omega}\phi D_{v}^{\sigma}f\partial^{\alpha-\sigma}v^{m+1}\,\diff x.
\end{multline*}
Again, using (\ref{eq:=000020alpha-Weak=000020=000020derivative})
we obtain 
\[
\int_{\Omega}D_{v}^{\sigma}f\partial^{\gamma}(\phi\partial^{\beta-\sigma}v^{m+1})\,\diff x=-\int_{\Omega}\phi\partial^{\beta-\sigma}v^{m+1}D_{v}^{\gamma+\sigma}f\,\diff x
\]
Thus 
\[
\int_{\Omega}fv^{m+1}\partial^{\alpha}\phi\,\diff x=(-1)^{|\alpha|}\sum_{0\leq\sigma\leq\beta}{\beta \choose \sigma}\int_{\Omega}\phi(D_{v}^{\rho}f\partial^{\alpha-\rho}v^{m+1}+D_{v}^{\sigma}f\partial^{\alpha-\sigma}v^{m+1})\,\diff x,
\]
where $\rho=\sigma+\gamma$. It follows that 
\[
\int_{\Omega}fv^{m+1}\partial^{\alpha}\phi\,\diff x=(-1)^{|\alpha|}\sum_{0\leq\sigma\leq\alpha}{\alpha \choose \sigma}\int_{\Omega}\phi(D_{v}^{\sigma}f\partial^{\alpha-\sigma}v^{m+1})\,\diff x,
\]
since 
\[
{\beta \choose \sigma-\gamma}+{\beta \choose \sigma}={\alpha \choose \sigma},
\]
Therefore, the map $W_{s,v}^{m,p}(\Omega)\ra W_{\mathrm{loc}}^{m,p}(\Omega)$,
$f\ra v^{m+1}f$ is well defined and (\ref{eq:Leibniz=000020form})
holds. Moreover, the continuity follows from (\ref{eq:Leibniz=000020form})
and (\ref{eq:=000020continuity}), which completes the proof.
\end{proof}

\section{\protect\label{sec:Density-of-smooth}Density of smooth functions}

In this section, we are interested in finding sufficient conditions
under which smooth functions are dense in weighted Sobolev spaces
$W_{s,v}^{m,p}(\Omega)$. 

Let $v\in\mathcal{C}^{m,*}(\Omega)$, and let $s:\pi_{m}\ra\R_{\geq0}$
 be a map satisfying (\ref{eq:=000020the=000020map=000020s}). In
order to prove the density of smooth functions in $W_{s,v}^{m,p}(\Omega)$,
we show first that 
\begin{equation}
\forall f\in W_{s,v}^{m,p}(\Omega)\,\forall n\in\N:\,f\chi_{n}\in W^{m,p}(\Omega),\label{eq:=000020multiplication=0000201}
\end{equation}
where $(\chi_{n})$ is the sequence constructed in Prop.~\ref{prop:comparison=000020with=000020weak=000020derivative=000020}.
After that, we show that the sequence $(f\chi_{n})$ converges to
$f$ in $W_{s,v}^{m,p}(\Omega)$. Then, we use the density results
of smooth functions in Sobolev spaces $W^{m,p}(\Omega)$ to conclude. 
\begin{thm}
\label{thm:=000020multiplication=000020chi_n}Let $\Omega\subseteq\R^{d}$
be an open set, $p\in[1,\infty)$, $m\in\N$, $v\in\mathcal{C}^{m,*}(\Omega)$,
$s:\pi_{m}\ra\R_{\geq0}$  be a map satisfying (\ref{eq:=000020the=000020map=000020s}).
Assume that 
\begin{equation}
\forall\alpha\in\pi_{m}^{0}\,\exists C>0\,\forall x\in M_{2n}^{c}\cap M_{n}:\,|\partial^{\alpha}v(x)|\leq C.\label{eq:bounded=000020derivatives=000020of=000020v}
\end{equation}
Then, the mapping $W_{s,v}^{m,p}(\Omega)\ra W^{m,p}(\Omega)$, $f\ra\chi_{n}f$
is continuous. In particular, for any $f\in W_{s,v}^{m,p}(\Omega)$
and for all $\alpha\in\pi_{m}^{0}$ we have 
\begin{equation}
D^{\alpha}(f\chi_{n})=\sum_{0\leq\beta\leq\alpha}{\alpha \choose \beta}D_{v}^{\beta}f\partial^{\alpha-\beta}\chi_{n}.\label{eq:D(f=000020=00005Cchi_n)}
\end{equation}
\end{thm}

\begin{proof}
First note that $\chi_{n}\in\mathcal{C}_{b}^{m}(\Omega)$ i.e.~$\partial^{\alpha}\chi_{n}$
is bounded on $\Omega$ for all $\alpha\in\pi_{m}$. This is because
$\eta_{n}^{(k)}\equiv0$ on $M_{2n}\cup M_{n}^{c}$ for all $k\neq0$,
and $\partial^{\alpha}v$ is bounded on $M_{2n}^{c}\cap M_{n}$. Let
$f\in W_{s,v}^{m,p}(\Omega)$. We show now that $D_{v}^{\beta}f\partial^{\alpha-\beta}\chi_{n}\in L^{p}(\Omega)$
for all $\alpha\in\pi_{m}$ and for all $\beta\in\pi_{\alpha}$. Indeed,
since supp$(\chi_{n})\subseteq M_{2n}^{c}$ we have 
\begin{multline}
||D_{v}^{\beta}f\partial^{\alpha-\beta}\chi_{n}||_{L^{p}(\Omega)}=|||v|^{-s(\beta)-1}\partial^{\alpha-\beta}\chi_{n}|v|^{s(\beta)+1}D_{v}^{\beta}f||_{L^{p}(\Omega)}\\
\leq(2n)^{s(\beta)+1}||\partial^{\alpha-\beta}\chi_{n}||_{L^{\infty}(\Omega)}|||v|^{s(\beta)+1}D_{v}^{\beta}f||_{L^{p}(\Omega)}.\label{eq:=000020continuity=000020multiplication=000020by=000020=00005Cchi_n}
\end{multline}
We claim now that $\chi_{n}f$ is $\alpha$-weakly differentiable
in the sense of $L_{\mathrm{loc}}^{1}(\Omega)$ for all $\alpha\in\pi_{m}^{0}$
and that (\ref{eq:D(f=000020=00005Cchi_n)}) holds. Indeed, we will
prove this by induction as we did for (\ref{eq:Leibniz=000020form}).
Suppose first that $\alpha\in\pi_{1}^{0}$. Choose $\phi\in\mathcal{D}(\Omega)$.
Using (\ref{eq:=000020alpha-Weak=000020=000020derivative}) with the
test function $\phi\chi_{n}v^{-|\alpha|-1}\in\mathcal{C}_{c}^{m}(\Omega)$
(because $\chi_{n}\equiv0$ on $M_{2n}$) we get 
\[
\int_{\Omega}f\partial^{\alpha}(\chi_{n}\phi)\,\diff x=-\int_{\Omega}\chi_{n}\phi D_{v}^{\alpha}f\,\diff x.
\]
It follows that 
\[
\int_{\Omega}f\chi_{n}\partial^{\alpha}\phi\,\diff x=\int_{\Omega}f(\partial^{\alpha}(\chi_{n}\phi)-\phi\partial^{\alpha}\chi_{n})\,\diff x=-\int_{\Omega}(\chi_{n}D_{v}^{\alpha}f+f\partial^{\alpha}\chi_{n})\phi\,\diff x,
\]
which proves our claim when $\alpha\in\pi_{1}^{0}$. We assume now
that (\ref{eq:D(f=000020=00005Cchi_n)}) holds for all $\alpha\in\pi_{l}^{0}$
for some $l<m$. Choose now $\alpha\in\pi_{l+1}^{l}$. Then $\alpha=\beta+\gamma$
for some $|\beta|=l$ and $|\gamma|=1$. Then for $\phi$ as above,
\begin{multline*}
\int_{\Omega}f\chi_{n}\partial^{\alpha}\phi\,\diff x=(-1)^{|\beta|}\int_{\Omega}D^{\beta}(f\chi_{n})\partial^{\gamma}\phi\,\diff x\\
=(-1)^{|\beta|}\sum_{0\leq\sigma\leq\beta}{\beta \choose \sigma}\int_{\Omega}D_{v}^{\sigma}f\partial^{\beta-\sigma}\chi_{n}\partial^{\gamma}\phi\,\diff x
\end{multline*}
where we used the induction assumption. Using the equality 
\[
\partial^{\beta-\sigma}\chi_{n}\partial^{\gamma}\phi=\partial^{\gamma}(\phi\partial^{\beta-\sigma}\chi_{n})-\phi\partial^{\alpha-\sigma}\chi_{n}
\]
we get 
\begin{multline*}
\int_{\Omega}f\chi_{n}\partial^{\alpha}\phi\,\diff x=(-1)^{|\beta|}\sum_{0\leq\sigma\leq\beta}{\beta \choose \sigma}\int_{\Omega}D_{v}^{\sigma}f\partial^{\gamma}(\phi\partial^{\beta-\sigma}\chi_{n})\,\diff x\\
+(-1)^{|\alpha|}\sum_{0\leq\sigma\leq\beta}{\beta \choose \sigma}\int_{\Omega}\phi D_{v}^{\sigma}f\partial^{\alpha-\sigma}\chi_{n}\,\diff x.
\end{multline*}
Again, using (\ref{eq:=000020alpha-Weak=000020=000020derivative})
we obtain 
\[
\int_{\Omega}D_{v}^{\sigma}f\partial^{\gamma}(\phi\partial^{\beta-\sigma}\chi_{n})\,\diff x=-\int_{\Omega}\phi\partial^{\beta-\sigma}\chi_{n}D_{v}^{\gamma+\sigma}f\,\diff x
\]
Thus 
\[
\int_{\Omega}f\chi_{n}\partial^{\alpha}\phi\,\diff x=(-1)^{|\alpha|}\sum_{0\leq\sigma\leq\beta}{\beta \choose \sigma}\int_{\Omega}\phi(D_{v}^{\rho}f\partial^{\alpha-\rho}\chi_{n}+D_{v}^{\sigma}f\partial^{\alpha-\sigma}\chi_{n})\,\diff x,
\]
where $\rho=\sigma+\gamma$. It follows that 
\[
\int_{\Omega}f\chi_{n}\partial^{\alpha}\phi\,\diff x=(-1)^{|\alpha|}\sum_{0\leq\sigma\leq\alpha}{\alpha \choose \sigma}\int_{\Omega}\phi(D_{v}^{\sigma}f\partial^{\alpha-\sigma}\chi_{n})\,\diff x,
\]
since 
\[
{\beta \choose \sigma-\gamma}+{\beta \choose \sigma}={\alpha \choose \sigma},
\]
Therefore, the map $W_{s,v}^{m,p}(\Omega)\ra W^{m,p}(\Omega)$, $f\ra\chi_{n}f$
is well defined and (\ref{eq:D(f=000020=00005Cchi_n)}) holds. Moreover,
the continuity follows from (\ref{eq:D(f=000020=00005Cchi_n)}) and
(\ref{eq:=000020continuity=000020multiplication=000020by=000020=00005Cchi_n}),
completing the proof.
\end{proof}
\begin{rem}
Using Thm.~\ref{thm:=000020multiplication=000020chi_n} with $m=0$
we get that $\chi_{n}f\in L^{p}(\Omega)$. By the dominated convergence
theorem we conclude that $\chi_{n}f\ra f$ in $L_{v}^{p}(\Omega)$
and hence $L^{p}(\Omega)$ is dense in $L_{v}^{p}(\Omega)$. If $v$
is bounded on $\Omega$, then $\mathcal{D}(\Omega)$ is dense in $L_{v}^{p}(\Omega)$. 

Another way to prove the density of $\mathcal{D}(\Omega)$ in $L_{w}^{p}(\Omega)$
($w\in E_{p}(\Omega)$) is to assume that $w\in\mathcal{C}^{\infty}(\Omega)$.
In this setting, given a sequence $(\phi_{n})$ of $\mathcal{D}(\Omega)$
that converges to $wf$ in $L^{2}(\Omega)$, the sequence $(\phi_{n}\chi_{n}w^{-1})$
($\chi_{n}$ defined with $v=w$) of $\mathcal{D}(\Omega)$ converges
to $f$ in $L_{w}^{p}(\Omega)$. 
\end{rem}

The aim now is to show that the sequence $(\chi_{n}f)$ of $W^{m,p}(\Omega)$
(provided assumption (\ref{eq:bounded=000020derivatives=000020of=000020v})
holds) converges to $f$ in $W_{s,v}^{m,p}(\Omega)$ when $v$ is
bounded. More precisely, we plan to show that 
\[
\forall\alpha\in\pi_{m}:\,|v|^{s(\alpha)+1}D^{\alpha}(f\chi_{n})=|v|^{s(\alpha)+1}\sum_{0\leq\beta\leq\alpha}{\alpha \choose \beta}D_{v}^{\beta}f\partial^{\alpha-\beta}\chi_{n}\to|v|^{s(\alpha)+1}D_{v}^{\alpha}f
\]
with convergence in $L^{p}(\Omega)$, where we used (\ref{eq:D(f=000020=00005Cchi_n)})
in the first equality. Since $\chi_{n}|v|^{s(\beta)+1}D_{v}^{\beta}f\ra|v|^{s(\beta)+1}D_{v}^{\beta}f$
in $L^{p}(\Omega)$ for all $\beta\in\pi_{m}$ by the dominated convergence
theorem, it suffices to show that 
\[
\forall\alpha\in\pi_{m}\,\forall\beta\in\pi_{\alpha}\backslash\{\alpha\}:\,|v|^{s(\beta)+1}(D_{v}^{\beta}f)\,|v|^{s(\alpha)-s(\beta)}\partial^{\alpha-\beta}\chi_{n}\ra0\,\,\text{in}\,\,L^{p}(\Omega).
\]
However, since
\begin{enumerate}
\item $\partial^{\sigma}\chi_{n}\ra0$ pointwise a.e. in $\Omega$ for all
$\sigma\in\pi_{m}^{0}$, 
\item $\,|v|^{s(\beta)+1}(D_{v}^{\beta}f)\in L^{p}(\Omega)$ for all $\beta\in\pi_{m}$,
\item $\partial^{\sigma}\chi_{n}=0$ on $M_{n}^{c}\cup M_{2n}$ and $|v|^{s(\sigma)}\leq n^{-s(\sigma)}$
on $M_{n}\cap M_{2n}^{c}$ for all $\sigma\in\pi_{m}^{0}$, 
\item $s(\alpha)-s(\beta)\geq s(\alpha-\beta)$ and hence $|v|^{s(\alpha)-s(\beta)}\leq C|v|^{s(\alpha-\beta)}$
for all $x\in\Omega$ for some $C$ independent of $x$ (since $v$
is bounded on $\Omega$) 
\end{enumerate}
it suffices to have the following property 
\begin{equation}
\exists C>0\,\forall n\in\N\,\forall\sigma\in\pi_{m}^{0}\,\forall x\in M_{n}\cap M_{2n}^{c}:\,|\partial^{\sigma}\chi_{n}|\leq Cn^{s(\sigma)}\label{eq:2.6}
\end{equation}
and conclude then by the dominated convergence theorem. 

In the next lemma we give an assumption on $v$ that yields (\ref{eq:2.6}). 
\begin{lem}
\label{lem:chi_n}Let $\Omega\subseteq\R^{d}$ be an open set, $p\in[1,\infty)$,
$m\in\N$, $v\in\mathcal{C}^{m,*}(\Omega)$, $s:\pi_{m}\ra\R_{\geq0}$
be a map satisfying (\ref{eq:=000020the=000020map=000020s}). Assume
that 
\begin{equation}
\exists C>0\,\forall n\in\N\,\forall\sigma\in\pi_{m}^{0}\,\forall x\in M_{n}\cap M_{2n}^{c}:\,|\partial^{\sigma}v(x)|\leq Cn^{s(\sigma)-1}.\label{eq:=000020assumption=000020on=000020w}
\end{equation}
Then, (\ref{eq:2.6}) holds.
\end{lem}

\begin{proof}
Fix $n\in\N$ and let $\sigma\in\pi_{m}^{0}$. By the Faà di Bruno
formula (\cite{wiki}), for any $x\in M_{n}\cap M_{2n}^{c}$ we have
\[
\partial^{\sigma}\chi_{n}(x)=\sum_{k=1}^{|\sigma|}|\{\pi\in\Pi\mid|\pi|=k\}|\eta_{n}^{(k)}(v(x))\sum_{\pi\in\Pi,|\pi|=k}\prod_{B\in\pi}\frac{\partial^{|B|}v}{\prod_{j\in B}\partial x_{j}}.
\]
(See the proof of \ref{enu:comp2} of Prop.~\ref{prop:comparison=000020with=000020weak=000020derivative=000020}
for more details). By construction of the sequence $(\eta_{n})$ we
have 
\[
\forall k\in\N\,\exists C>0\,\forall n\in\N\,\forall x\in\R:\,|\eta_{n}^{(k)}(x)|\leq Cn^{k}.
\]
For any $B\in\pi$, we associate $\beta\in\N^{d}$ by setting $\beta_{j}:=|\{j\in B\}|$
for all $j$. It follows that 
\begin{equation}
|\partial^{\sigma}\chi_{n}(x)|\leq C\sum_{k=1}^{|\sigma|}C_{k}n^{k}\sum_{\pi\in\Pi,|\pi|=k}\prod_{B\in\pi}n^{s(\beta)-1}\label{eq:=000020faa=000020di=000020bruno=0000202}
\end{equation}
where we used assumption (\ref{eq:=000020assumption=000020on=000020w}).
It is easy to see that 
\[
\forall\pi\in\Pi:\,\prod_{B\in\pi}n^{s(\beta)-1}=n^{\sum_{B\in\pi}s(\beta)-|\pi|}\leq n^{s(\sigma)-|\pi|}
\]
where we used (\ref{eq:=000020the=000020map=000020s}), which gives
(\ref{eq:2.6}) when used in (\ref{eq:=000020faa=000020di=000020bruno=0000202}). 
\end{proof}
\begin{thm}
\label{thm:density}Let $\Omega\subseteq\R^{d}$ be an open set, $p\in[1,\infty)$,
$m\in\N$, and $v\in\mathcal{C}^{m,*}(\Omega)\cap L^{\infty}(\Omega)$,
$s:\pi_{m}\ra\R_{\geq0}$ be a map satisfying (\ref{eq:=000020the=000020map=000020s}).
Assume that (\ref{eq:=000020assumption=000020on=000020w}) holds.
Then, for any $f\in W_{s,v}^{m,p}(\Omega)$, the sequence $(\chi_{n}f)$
of $W^{m,p}(\Omega)$ (by Thm.~\ref{thm:=000020multiplication=000020chi_n})
converges to $f$ in $W_{s,v}^{m,p}(\Omega)$. 
\end{thm}

\begin{proof}
It follows from Lem.~\ref{lem:chi_n} and the paragraph preceding
it.
\end{proof}
We deduce from this theorem the following density results 
\begin{cor}
\label{cor:density}Under the assumptions of Thm.~\ref{thm:density},
we have 
\begin{enumerate}
\item \label{enu:Omega=00003DR^d}The space $\mathcal{D}(\Omega)$ is dense
in $W_{s,v}^{m,p}(\Omega)$ when $\Omega=\R^{d}$.
\item \label{enu:omega=000020bounded=000020}The space $\mathcal{C}^{\infty}(\Omega)\cap W^{m,p}(\Omega)$
is dense in $W_{s,v}^{m,p}(\Omega)$ in case $\Omega$ is bounded.
\end{enumerate}
\end{cor}

\begin{proof}
Let $f\in W_{s,v}^{m,p}(\R^{d})$. By Thm.~\ref{thm:density}, the
sequence $(\chi_{n}f)$ of $W^{m,p}(\R^{d})$ converges to $f$ in
$W_{s,v}^{m,p}(\R^{d})$. By density of $\mathcal{D}(\R^{d})$ in
$W^{m,p}(\R^{d})$, for all $n$ there exists a sequence $(\phi_{n,k})_{k}$
of $\mathcal{D}(\R^{d})$ that converges to $\chi_{n}f$ in $W^{m,p}(\R^{d})$,
and hence in $W_{s,v}^{m,p}(\R^{d})$ since $v\in L^{\infty}(\R^{d})$,
which proves assertion \ref{enu:Omega=00003DR^d}. 

Similarly, to prove assertion \ref{enu:omega=000020bounded=000020},
we use Thm.~\ref{thm:density} and the fact that the space $\mathcal{C}^{\infty}(\Omega)\cap W^{m,p}(\Omega)$
is dense in $W^{m,p}(\Omega)$ when $\Omega$ is bounded (see Thm.~2
of Subsection $5.3.2$ of \cite{Evans}).
\end{proof}
\begin{rem}
The boundedness assumption of $v$ in Thm.~\ref{thm:density} can
be dropped in case $s(\alpha)=|\alpha|$ since the assumption was
used to guarantee that the constant $C$ in (\ref{eq:=000020w^=00005Calpha<=000020w_=00005Calpha})
is independent of $K$.
\end{rem}

\section{\protect\label{sec:The-trace-operator}The trace operator}

Now, we discuss the different possibilities to define a trace operator
on $W_{s,v}^{1,p}(\Omega)$
\begin{rem}
\label{rem:trace}Let $\Omega\subseteq\R^{d}$ be a bounded open set
with a sufficiently smooth boundary\footnote{so that the trace operator on $W^{1,p}(\Omega)$ is well defined.},
e.g.~$\partial\Omega$ is $\mathcal{C}^{1}$, $p\in[1,\infty)$,
$v\in\mathcal{C}^{1,*}(\Omega)$, $s:\pi_{1}\ra\R_{\geq0}$ be a map
satisfying (\ref{eq:=000020the=000020map=000020s}). 
\begin{enumerate}
\item \label{enu:First-possibility}Assume that 
\begin{equation}
\exists K\Subset\Omega\,\exists\delta>0\,\forall x\in\Omega\backslash K:\,|v(x)|>\delta\label{eq:2.9}
\end{equation}
(or equivalently $M_{n_{0}}\Subset\Omega$ for some $n_{0}\in\N$).
Fix $n\geq n_{0}$. Then, (\ref{eq:bounded=000020derivatives=000020of=000020v})
holds and $\partial\Omega\subseteq\partial M_{n}^{c}$. By Thm.~\ref{thm:=000020multiplication=000020chi_n}
we have 
\begin{equation}
\chi_{n}f\in W^{1,p}(\Omega)\,\,\,\text{and}\,\,\,\forall\alpha\in\pi_{1}^{0}:\,D^{\alpha}(\chi_{n}f)=f\partial^{\alpha}\chi_{n}+\chi_{n}D_{v}^{\alpha}f.\label{eq:trace=0000201=000020possibility}
\end{equation}
If we denote by $|_{\partial\Omega}:W^{1,p}(\Omega)\ra L^{p}(\partial\Omega)$
the trace operator, then $(f\chi_{n})|_{\partial\Omega}$ is well
defined. Since $\chi_{n}\equiv1$ on $\Omega\backslash K\subseteq M_{n}^{c}$,
we have that $(f\chi_{n})|_{\partial\Omega}=f|_{\partial\Omega}$
which implies that $|_{\partial\Omega}$ can be extended to the space
$W_{s,v}^{1,p}(\Omega)$. Furthermore, by Thm.~\ref{thm:=000020multiplication=000020chi_n}
and the trace theorem in $W^{1,p}(\Omega)$ we get 
\begin{equation}
||f|_{\partial\Omega}||_{L^{p}(\partial\Omega)}\leq C||\chi_{n}f||_{W^{1,p}(\Omega)}\leq C_{n}||f||_{W_{s,v}^{1,p}(\Omega)}\label{eq:cont=000020trace=0000201}
\end{equation}
where $C_{n}$ is a constant that depends on $n$ but not in $f$,
which proves that $|_{\partial\Omega}$ is continuous on $W_{s,v}^{1,p}(\Omega)$. 
\item \label{enu:Second-possibility}If we assume that $v\in\mathcal{C}_{b}^{1}(\Omega)$,
then by Prop.~\ref{prop:continuous=000020embedding=000020W^mp},
we have that the map $W_{s,v}^{1,p}(\Omega)\ra W^{1,p}(\Omega)$,
$f\ra v^{2}f$ is continuous. In particular $(v^{2}f)|_{\partial\Omega}$
is well defined and belongs to $L^{p}(\partial\Omega)$. Moreover,
by the trace theorem in $W^{1,p}(\Omega)$ we have 
\begin{equation}
||(v^{2}f)|_{\partial\Omega}||_{L^{p}(\partial\Omega)}\leq C||v^{2}f||_{W^{1,p}(\Omega)}\leq C||f||_{W_{s,v}^{1,p}(\Omega)}\label{eq:=000020cont=000020trace=0000202}
\end{equation}
Therefore, the operator $W_{s,v}^{1,p}(\Omega)\ra L^{p}(\partial\Omega)$,
$f\ra(v^{2}f)|_{\partial\Omega}$ is continuous. 
\end{enumerate}
Therefore, we can define the trace operator on the space $W_{s,v}^{1,p}(\Omega)$
either by setting $\text{Tr}(f)=f|_{\partial\Omega}\in L^{p}(\partial\Omega)$,
or by setting $\text{Tr}(f)=(v^{2}f)|_{\partial\Omega}\in L^{p}(\partial\Omega)$
depending on the assumptions satisfied by the function $v$. 
\end{rem}

\begin{rem}
\label{rem:trace=000020}Note that in the favorite situation in which
assumption (\ref{eq:2.9}) holds and $v\in\mathcal{C}_{b}^{1}(\Omega)$,
both $f|_{\partial\Omega}$ and $(v^{2}f)|_{\partial\Omega}$ are
well defined and belong to $L^{p}(\partial\Omega)$ for all $f\in W_{s,v}^{1,p}(\Omega)$.
Moreover, both $v$ and $v^{-1}$ are bounded on $\partial\Omega$,
which implies that the operators $W_{s,v}^{1,p}(\Omega)\ra L^{p}(\partial\Omega)$,
$f\ra f|_{\partial\Omega}$, $f\ra(fv^{2})|_{\partial\Omega}$ are
equivalent, and hence $f|_{\partial\Omega}=0$ if and only if $(v^{2}f)|_{\partial\Omega}=0$. 
\end{rem}

These different possibilities motivate the following definition
\begin{defn}
Let $\Omega\subseteq\R^{d}$ be a bounded open set and $\partial\Omega$
is $\mathcal{C}^{1}$, $p\in[1,\infty)$, $v\in\mathcal{C}^{1,*}(\Omega)$,
$s:\pi_{1}\ra\R_{\geq0}$ be a map satisfying (\ref{eq:=000020the=000020map=000020s}). 
\begin{enumerate}
\item If assumption (\ref{eq:2.9}) holds, then 
\begin{equation}
\forall f\in W_{s,v}^{1,p}(\Omega):\,\text{Tr}_{1}(f):=f|_{\partial\Omega}\in L^{p}(\partial\Omega)\label{eq:trace=0000201}
\end{equation}
\item If $v\in\mathcal{C}_{b}^{1}(\Omega)$, then 
\begin{equation}
\forall f\in W_{s,v}^{1,p}(\Omega):\,\text{Tr}_{2}(f):=(v^{2}f)|_{\partial\Omega}\in L^{p}(\partial\Omega).\label{eq:trace=0000202}
\end{equation}
\end{enumerate}
\end{defn}

As we have seen in the beginning of this section, under the assumptions
given in the latter definition, the trace operators are continuous
(see (\ref{eq:cont=000020trace=0000201}) and (\ref{eq:=000020cont=000020trace=0000202})).

\section{\protect\label{sec:=000020space=000020of=000020null=000020trace=000020functions=000020}Spaces
of null trace}
\begin{defn}
Let $\Omega\subseteq\R^{d}$ be an open set, $p\in[1,\infty)$, $V:=\{w_{\alpha}\}_{\alpha\in\pi_{1}}$
be a collection of elements of $E_{p}(\Omega)$, and let $v\in\mathcal{C}^{1,*}(\Omega)$
be such that (\ref{eq:=000020w^=00005Calpha<=000020w_=00005Calpha})
holds; $X_{V,v,0}^{1,p}(\Omega)$ denotes the closure of $\mathcal{D}(\Omega)$
in $W_{V,v}^{1,p}(\Omega)$.
\end{defn}

The assumption (\ref{eq:=000020w^=00005Calpha<=000020w_=00005Calpha})
guarantees that the weighted Sobolev space $W_{V,v}^{1,p}(\Omega)$
is Cauchy complete, and hence $X_{V,v,0}^{1,p}(\Omega)$ is well defined. 

We now define the spaces $W_{s,v,0}^{1,p,i}$ ($i=1,$ $2$) using
the trace operators
\begin{defn}
Let $\Omega$ be a bounded open set and $\partial\Omega$ is $\mathcal{C}^{1}$,
$p\in[1,\infty)$, $v\in\mathcal{C}^{1,*}(\Omega)$, $s:\pi_{1}\ra\R_{\geq0}$
 be a map satisfying (\ref{eq:=000020the=000020map=000020s}). 
\begin{enumerate}
\item Assume that (\ref{eq:2.9}) holds. Then we set 
\[
W_{s,v,0}^{1,p,1}(\Omega):=\{f\in W_{s,v}^{1,p}(\Omega)\mid\text{Tr}_{1}(f)=0\}.
\]
\item Assume that $v\in\mathcal{C}_{b}^{1}(\Omega)$. Then we set 
\[
W_{s,v,0}^{1,p,2}(\Omega):=\{f\in W_{s,v}^{1,p}(\Omega)\mid\text{Tr}_{2}(f)=0\}.
\]
 
\end{enumerate}
\end{defn}

\begin{rem}
\label{rem:=000020space=000020of=000020null=000020trace=000020}~Recall
from Rem.~\ref{rem:trace=000020} that the operators $\text{Tr}_{1}$,
$\text{Tr}_{2}$ are equivalent when $v\in\mathcal{C}_{b}^{1}(\Omega)$
and (\ref{eq:2.9}) holds, which implies that 
\[
W_{s,v,0}^{1,p,1}(\Omega)=W_{s,v,0}^{1,p,2}(\Omega).
\]
\end{rem}

\begin{prop}
\label{prop:Banach=000020space=000020}Let $\Omega$ be a bounded
open set and $\partial\Omega$ is $\mathcal{C}^{1}$, $p\in[1,\infty)$,
$v\in\mathcal{C}^{1,*}(\Omega)$, $s:\pi_{1}\ra\R_{\geq0}$ be a map
satisfying (\ref{eq:=000020the=000020map=000020s}). 
\begin{enumerate}
\item If (\ref{eq:2.9}) holds, then the space $W_{s,v,0}^{1,p,1}(\Omega)$
is Cauchy complete.
\item If $v\in\mathcal{C}_{b}^{1}(\Omega)$, then the space $W_{s,v,0}^{1,p,2}(\Omega)$
is Cauchy complete.
\end{enumerate}
\end{prop}

\begin{proof}
It is sufficient to prove that the space $W_{s,v,0}^{1,p,i}(\Omega)$
($i=1$, $2$) is closed since it is a subspace of a Banach space
$W_{s,v}^{1,p}(\Omega)$. However, the closedness follows from the
fact that $W_{s,v,0}^{1,p,i}(\Omega)=\mathrm{Ker}(\mathrm{Tr}_{i})$
and $\text{Tr}_{i}$ is continuous. 
\end{proof}
\begin{thm}
\label{thm:=000020equality=000020of=000020spaces=000020of=000020null=000020trace=000020}Let
$\Omega$ be a bounded open set and $\partial\Omega$ is $\mathcal{C}^{1}$,
$p\in[1,\infty)$, $v\in\mathcal{C}^{1,*}(\Omega)$, $s:\pi_{1}\ra\R_{\geq0}$
 be a map satisfying (\ref{eq:=000020the=000020map=000020s}). 
\begin{enumerate}
\item If $v$ is bounded and (\ref{eq:2.9}) holds, then $X_{s,v,0}^{1,p}(\Omega)=W_{s,v,0}^{1,p,1}(\Omega).$
\item If $v\in\mathcal{C}_{b}^{1}(\Omega)$ then $X_{s,v,0}^{1,p}(\Omega)=W_{s,v,0}^{1,p,2}(\Omega).$ 
\end{enumerate}
\end{thm}

\begin{proof}
For $i=1$, $2$, the inclusion $X_{s,v,0}^{1,p}(\Omega)\subseteq W_{s,v,0}^{1,p,i}(\Omega)$
follows from the fact that $\text{Tr}_{i}(\phi)=0$ for all $\phi\in\mathcal{D}(\Omega)$,
and $\text{Tr}_{i}$ is continuous. 

We deal now with the opposite inclusion. Claim first that $\chi_{n}f\in W_{0}^{1,p}(\Omega)$
and that the sequence $(\chi_{n}f)$ converges to $f$ in $W_{s,v}^{1,p}(\Omega)$
for all $f\in W_{s,v,0}^{1,p,i}(\Omega)$ ($i=1$, $2$). Indeed,
in case (\ref{eq:2.9}) holds, the first part of the claim follows
from the definition of the space $W_{s,v,0}^{1,p,1}(\Omega)$ (see
also \ref{enu:First-possibility} of Rem.~\ref{rem:trace}). In case
$v\in\mathcal{C}_{b}^{1}(\Omega)$, $\chi_{n}f\in W^{1,p}(\Omega)$
by Thm.~(\ref{thm:=000020multiplication=000020chi_n}). Moreover,
we have that $(f\chi_{n})|_{\partial M_{2n}\cap\partial\Omega}=0$
because $\chi_{n}\equiv0$ on $M_{2n}$, and $(f\chi_{n})|_{\partial M_{2n}^{c}\cap\partial\Omega}=(fv^{2}\chi_{n}v^{-2})|_{\partial M_{2n}^{c}\cap\partial\Omega}=0$
because $|v^{-2}|\leq(2n)^{2}$ on $M_{2n}^{c}$, which proves that
$f\chi_{n}\in W_{0}^{1,p}(\Omega)$, and hence the first part of the
claim is proved. The second part of the claim is a consequence of
Thm.~\ref{thm:density}. Finally, by density of $\mathcal{D}(\Omega)$
in $W_{0}^{1,p}(\Omega)$, and since $v$ is bounded, we conclude
that $\mathcal{D}(\Omega)$ is dense $W_{s,v,0}^{1,p,i}(\Omega)$. 
\end{proof}
In the next lemma, we consider a functional of the form $X_{s,v,0}^{1,p}(\Omega)\ra\R$,
$f\ra\int_{\Omega}gD^{\alpha}(v^{3}f)\,\diff x$ where $v\in\mathcal{C}^{1,*}(\Omega)$,
$s:\pi_{1}\ra\R_{\geq0}$ is a map satisfying (\ref{eq:=000020the=000020map=000020s}),
and $g\in W_{s,v}^{1,p'}(\Omega)$ ($p'$ is the exponent conjugate
to $p$). We show in particular that 
\begin{equation}
\forall f\in X_{s,v,0}^{1,p}(\Omega)\,\forall\alpha\in\pi_{1}^{0}:\,\int_{\Omega}gD^{\alpha}(v^{3}f)\,\diff x=-\int_{\Omega}v^{3}fD_{v}^{\alpha}g\,\diff x\label{eq:IPP=0000202}
\end{equation}
which implies that the functional $f\ra\int_{\Omega}gD^{\alpha}(v^{3}f)\,\diff x$
is continuous. As far as we know, the latter functional and the integration
by part formula (\ref{eq:IPP=0000202}) was not considered in the
literature. 
\begin{lem}
\label{lem:Ibp}Let $\Omega\subseteq\R^{d}$ be a bounded open set,
$v\in\mathcal{C}^{1,*}(\Omega)$, $s:\pi_{1}\ra\R_{\geq0}$  be a
map satisfying (\ref{eq:=000020the=000020map=000020s}). Assume that 
\begin{enumerate}
\item \label{enu:=000020IBP1}$v\in L^{\infty}(\Omega)$ 
\item $a\in L_{v^{-3}}^{\infty}(\Omega)$ with $av^{-2}\in\mathcal{C}^{1}(\Omega)$
and $\partial^{\alpha}a\in L_{v^{-2}}^{\infty}(\Omega)$ for all $\alpha\in\pi_{1}^{0}$.
\end{enumerate}
Then 
\[
\forall q\in[1,\infty)\,\forall g\in W_{s,v}^{1,q}(\Omega):\,ag\in W^{1,q}(\Omega)\,\,\,\text{and}\,\,\,D^{\alpha}(ag)=g\partial^{\alpha}a+aD_{v}^{\alpha}g.
\]
Moreover,
\begin{equation}
\forall h\in W_{s,v}^{1,p}(\Omega)\,\forall f\in X_{s,v,0}^{1,p'}(\Omega):\int_{\Omega}hD^{\alpha}(af)\,\diff x=-\int_{\Omega}afD_{v}^{\alpha}h\,\diff x\label{eq:IPP}
\end{equation}
for all $p\in(1,\infty)\,$, where $p'$ is the exponent conjugate
to $p$. 
\end{lem}

\begin{proof}
The proof of the first part of the lemma is very similar to the proof
of Thm.~\ref{prop:continuous=000020embedding=000020W^mp}. We prove
now (\ref{eq:IPP}). By definition of $X_{s,v,0}^{1,p'}$, there exists
a sequence $(\phi_{n})$ of $\mathcal{D}(\Omega)$ that converges
to $f$ in $W_{s,v}^{1,p'}(\Omega)$. Since $h\in W_{s,v}^{1,p}(\Omega)$,
we have 
\[
\int_{\Omega}h\phi_{n}\partial^{\alpha}a\,\diff x+\int_{\Omega}ha\partial^{\alpha}\phi_{n}\,\diff x=\int_{\Omega}h\partial^{\alpha}(a\phi_{n})\,\diff x=-\int_{\Omega}a\phi_{n}D_{v}^{\alpha}h\,\diff x
\]
\[
\,
\]
for all $n$ and for all $\alpha\in\pi_{1}^{0}$, where we simply
used the definition of the weak derivative in the sense of $L_{v,\mathrm{loc}}^{1}(\Omega)$
with the test function $av^{-2}\phi_{n}\in\mathcal{C}_{0}^{1}(\Omega)$.
By letting $n$ tends to $\infty$ we get 
\begin{equation}
\int_{\Omega}hf\partial^{\alpha}a\,\diff x+\int_{\Omega}haD_{v}^{\alpha}f\,\diff x=-\int_{\Omega}afD_{v}^{\alpha}h\,\diff x\label{eq:2.12}
\end{equation}
since 
\[
\left|\int_{\Omega}h\partial^{\alpha}a(\phi_{n}-f)\,\diff x\right|\leq C||\partial^{\alpha}a||_{L_{v^{-2}}^{\infty}(\Omega)}||hv||_{L^{p}(\Omega)}||v(\phi_{n}-f)||_{L^{p'}(\Omega)},
\]
\[
\left|\int_{\Omega}ha(\partial^{\alpha}\phi_{n}-D_{v}^{\alpha}f)\,\diff x\right|\leq C||a||_{L_{v^{-3}}^{\infty}(\Omega)}||hv||_{L^{p}(\Omega)}||(\partial^{\alpha}\phi_{n}-D_{v}^{\alpha}f)|v|^{s(\alpha)+1}||_{L^{p'}(\Omega)},
\]
and 
\[
\left|\int_{\Omega}a(\phi_{n}-f)D_{v}^{\alpha}h\,\diff x\right|\leq C||a||_{L_{v^{-3}}^{\infty}(\Omega)}|||v|^{s(\alpha)+1}D_{v}^{\alpha}h||_{L^{p}(\Omega)}||v(\phi_{n}-f)||_{L^{p'}(\Omega)}.
\]
Using now (\ref{eq:2.12}) and the first part of the lemma with $g=f$
we obtain (\ref{eq:IPP}). 
\end{proof}
\begin{rem}
~
\begin{enumerate}
\item When $v\in\mathcal{C}_{b}^{1}(\Omega)$, the function $a$ in the
previous lemma can be any function of the form $a=\tilde{a}v^{n}$
where $\tilde{a}\in\mathcal{C}_{b}^{1}(\Omega)$ and $n\in\N$ with
$n\geq3$. 
\item Clearly, the first part of the lemma can be proved with less assumptions
on $\Omega$, $v$ and $a$. 
\end{enumerate}
\end{rem}

\section{\protect\label{sec:Poincar=0000E9-inequality}Poincaré inequality}

In this section we are interesting in proving a Poincaré inequality
for the weighted Sobolev space $X_{s,v,0}^{1,p}(\Omega)$. 

We first derive an interesting inequality where we prove that the
term $||f\nabla(v^{2})||_{L^{p}}$ is bounded by $||v^{2}D_{v}f||_{L^{p}}$.
We use then this inequality to prove an interesting Poincaré inequality
under the assumption $v\in\mathcal{C}_{b}^{1,*}(\Omega):=\mathcal{C}_{b}^{1}(\Omega)\cap\mathcal{C}^{1,*}(\Omega)$. 
\begin{thm}
\label{thm:Kebiche=000020inequality=000020}Let $\Omega\subseteq\R^{d}$
($d\geq2$) be a bounded open set, $p\in[1,d)$, $v\in\mathcal{C}_{b}^{1,*}(\Omega)$,
$s:\pi_{1}\ra\R_{\geq0}$ be a map satisfying (\ref{eq:=000020the=000020map=000020s}).
Assume that 
\begin{equation}
\exists\sigma>0\,\forall x\neq0:\,|\nabla v(x)|_{p}:=\left(\sum_{i=1}^{d}|\partial_{i}v(x)|^{p}\right)^{1/p}\leq\sigma\frac{|v(x)|}{|x|}\label{eq:=000020w'<w/x}
\end{equation}
where $|x|$ is the euclidean norm of $\R^{d}$, and that 
\begin{equation}
0<\frac{2\sigma p}{d-p}<1.\label{eq:dimension}
\end{equation}
Then, 
\begin{equation}
\forall f\in X_{s,v,0}^{1,p}(\Omega):\,||f\nabla(v^{2})||_{L^{p}}\leq\frac{2\sigma p}{d-p-2\sigma p}||v^{2}D_{v}f||_{L^{p}}.\label{eq:Kebiche=000020inequality}
\end{equation}
\end{thm}

\begin{proof}
Let $f\in X_{s,v,0}^{1,p}(\Omega)$. By Thm.~\ref{thm:=000020equality=000020of=000020spaces=000020of=000020null=000020trace=000020}
we have 
\[
fv^{2}\in W_{0}^{1,p}(\Omega)\,\text{with }D^{\alpha}(fv^{2})=v^{2}D_{v}^{\alpha}f+f\partial^{\alpha}v^{2},\,\,\,\forall\alpha\in\pi_{1}^{0}.
\]
(See also Prop.~\ref{prop:continuous=000020embedding=000020W^mp}.)
Using (\ref{eq:=000020w'<w/x}) and Hardy inequality we get 
\begin{align*}
||f\nabla(v^{2})||_{L^{p}(\Omega)} & :=2\left(\sum_{i=1}^{d}||fv\partial_{i}v||_{L^{p}(\Omega)}^{p}\right)^{\frac{1}{p}}\\
 & \leq2\sigma\left\Vert \frac{fv^{2}}{x}\right\Vert _{L^{p}(\Omega)}\\
 & \leq2\sigma\frac{p}{d-p}\left\Vert D(fv^{2})\right\Vert _{L^{p}(\Omega)}\\
 & \leq2\sigma\frac{p}{d-p}\left\Vert v^{2}D_{v}f\right\Vert _{L^{p}(\Omega)}+2\sigma\frac{p}{d-p}\left\Vert f\nabla(v^{2})\right\Vert _{L^{p}(\Omega)}
\end{align*}
This inequality together with (\ref{eq:dimension}) gives (\ref{eq:Kebiche=000020inequality}). 
\end{proof}
\begin{example}
The typical example of a function satisfying (\ref{eq:=000020w'<w/x})
when $p\leq2$ is the function $v(x)=|x|^{\beta}$ with $\beta\geq2$,
where $\sigma=\beta d^{\frac{2-p}{2}}$.
\end{example}

Similarly, we have 
\begin{thm}
Let $\Omega\subseteq\R$ be an open set, $p\in(1,\infty)$, $v\in\mathcal{C}^{1,*}(\Omega)$
with $v'\in L^{\infty}(\Omega)$, and let $s:\pi_{1}\ra\R_{\geq0}$
be a map satisfying (\ref{eq:=000020the=000020map=000020s}). Assume
that 
\[
\exists\sigma>0\,\forall x\neq0:\,|v'(x)|\leq\sigma\frac{|v(x)|}{|x|},\,\,\,0<\frac{2\sigma p}{p-1}<1.
\]
Then 
\[
\forall f\in X_{s,v,0}^{1,p}(\Omega):\,||f(v^{2})'||_{L^{p}(\Omega)}\leq\frac{2\sigma p}{p-1-2\sigma p}||v^{2}D_{v}f||_{L^{p}(\Omega)}.
\]
\end{thm}

The proof is identical to the proof of Thm.~\ref{thm:Kebiche=000020inequality=000020}. 

Using Poincaré inequality and Thm.~\ref{thm:Kebiche=000020inequality=000020}
we get 
\begin{cor}
\label{cor:Poincar=0000E9=000020inequality=000020}Let $\Omega\subseteq\R^{d}$
($d\geq2$) be a bounded open set, $p\in[1,d)$, $v\in\mathcal{C}_{b}^{1,*}(\Omega)$,
$s:\pi_{1}\ra\R_{\geq0}$ be a map satisfying (\ref{eq:=000020the=000020map=000020s}).
Assume that (\ref{eq:=000020w'<w/x}) and (\ref{eq:dimension}) hold.
Then, 
\[
\exists C_{\Omega}>0\,\forall f\in X_{s,v,0}^{1,p}:\,||v^{2}f||_{L^{p}}\leq C_{\Omega}||v^{2}D_{v}f||_{L^{p}}
\]

\end{cor}

\begin{proof}
This is a direct consequence of Poincaré inequality in $W_{0}^{1,p}(\Omega)$
and Thm.~\ref{thm:Kebiche=000020inequality=000020}. 
\end{proof}

\section{\protect\label{sec:Degenerate-elliptic-PDE}Degenerate elliptic PDE}

In the sequel, if $\alpha\in\pi_{1}^{0}$, we write $D_{v}^{i}$ instead
of $D_{v}^{\alpha}$ where $i$ is the index of the non-null component
of $\alpha$. We also write $D_{v}$ for the operator $(D_{v}^{i})$. 

We consider in this section, a boundary-value problem of the form
\begin{equation}
\begin{cases}
-\sum_{i,j=1}^{d}D_{j}(a_{ij}D_{i}f)+(b\cdot D)f+cf=h & \text{in }\Omega\\
f=0 & \text{on }\partial\Omega
\end{cases}\label{eq:First=000020PDE}
\end{equation}
where the coefficients $(a_{ij})$, $b=(b_{i})$, $c$ are given functions,
and $h$ is some given functional. In particular, we consider that
the bilinear form $\sum_{i,j=1}^{d}a_{ij}(x)\xi_{i}\xi_{j}$ degenerates
i.e.~is not uniformly elliptic. Let $V:=\{w_{\alpha}\}_{\alpha\in\pi_{1}}$
be a collection of elements of $E_{p}(\Omega)$, and $v\in\mathcal{C}^{1,*}(\Omega)$
that satisfy (\ref{eq:=000020w^=00005Calpha<=000020w_=00005Calpha}),
and $h$ be a continuous functional on $X_{V,v,0}^{1}(\Omega)$ ($:=X_{V,v,0}^{1,2}(\Omega)$).
We say that $f\in X_{V,v,0}^{1}(\Omega)$ is a weak solution to the
boundary-value problem (\ref{eq:First=000020PDE}) if 
\begin{equation}
\forall g\in X_{V,v,0}^{1}(\Omega):\,B_{v}(f,g)=h(g)\label{eq:weak=000020formulation}
\end{equation}
where $B_{v}(\cdot,\cdot):X_{V,v,0}^{1}(\Omega)\times X_{V,v,0}^{1}(\Omega)\ra\R$
is the bilinear form given by 
\begin{equation}
B_{v}(f,g):=\sum_{i,j=1}^{d}\int_{\Omega}a_{ij}D_{v}^{i}fD_{v}^{j}g\,\diff x+\int_{\Omega}g(b\cdot D_{v})f\,\diff x+\int_{\Omega}cfg\,\diff x.\label{eq:Bilnear=000020form}
\end{equation}

Assume that $f\in X_{V,v,0}^{1}(\Omega)$ is a weak solution to (\ref{eq:First=000020PDE})
and assume that 
\begin{equation}
\forall i,j=1,...,d:\,a_{ij},\,b_{i},\,c\in L_{\mathrm{loc}}^{\infty}(\Omega_{|v|^{2p}}^{*})\,\text{ and }\,h\in L_{\mathrm{loc}}^{2}(\Omega_{|v|^{2p}}^{*}).\label{eq:local=000020coeff}
\end{equation}
Then the differential equation of (\ref{eq:First=000020PDE}) holds
in the sense of $\mathcal{D}'(\Omega_{|v|^{2p}}^{*})$. Indeed, we
know that $X_{V,v,0}^{1}(\Omega)\subseteq W^{1,2}(\Omega_{|v|^{2p}}^{*},V)$.
Hence, 
\[
\forall i,j=1,...,d:\,a_{ij}D_{i}f,\,b_{i}D_{i}f,\,cf,\,h\in L_{\text{loc}}^{1}(\Omega_{|v|^{2p}}^{*})\subseteq\mathcal{D}'(\Omega_{|v|^{2p}}^{*}).
\]
The remaining part of the proof is trivial. 

For an in-depth study of weak solution existence to boundary value
problems of the form (\ref{eq:First=000020PDE}), we refer to \cite{KuAn}
and to reference therein. In the next theorem, we give sufficient
assumptions that guarantee the existence of a unique weak solution
in $X_{s,v,0}^{1}(\Omega)$ to the boundary-value problem (\ref{eq:First=000020PDE}).
We use in particular the Poincaré inequality proved in Cor.~\ref{cor:Poincar=0000E9=000020inequality=000020}.
Moreover, it is possible to take $h(g)=\int_{\Omega}k\nabla(v^{3}g)\,\diff x$
for all $g\in X_{s,v,0}^{1}(\Omega)$ where $k\in W_{s,v}^{1,2}(\Omega)$
(see Lem.~\ref{lem:Ibp}). 

For simplicity, we will limit ourselves to the case where $s(\cdot)=|\cdot|$
(in particular $w_{\alpha}=v^{|\alpha|+1}\in\mathcal{C}^{1}(\Omega)$
for all $\alpha$). We will then write $X_{|\cdot|,v,0}^{1}(\Omega)$
instead of $X_{s,v,0}^{1}(\Omega)$. A generalization to an arbitrary
map $s$ satisfying (\ref{eq:=000020the=000020map=000020s}), is possible
with some small modification of the assumptions. 
\begin{thm}
\label{thm:second=000020application=000020}Let $\Omega\subseteq\R^{d}$
be a bounded open set, $v\in\mathcal{C}^{1,*}(\Omega)$, $(a_{ij})$
be such that $a_{ij}\in L_{v^{-4}}^{\infty}(\Omega)$ and satisfies
\begin{equation}
\exists\mu>0\,\forall\xi\in\R^{d}:\,\mu v^{4}(x)|\xi|^{2}\le\sum_{i,j=1}^{d}a_{ij}(x)\xi_{i}\xi_{j}\,\,\,\mathrm{a.e.}\,\text{in }\,\Omega.\label{eq:ellipticity=000020condition=000020w}
\end{equation}
Let $c$ be such that $v^{-2}c\in L^{\infty}(\Omega)$, and $h$ a
continuous functional on $X_{|\cdot|,v,0}^{1}(\Omega)$. Suppose furthermore
that one of the following assumptions holds
\begin{enumerate}
\item \label{enu:=0000201} $v^{-3}b_{i}\in L^{\infty}(\Omega)$ for all
$i=1,...,d$, $cv^{-2}\geq\sigma>0$ a.e. in $\Omega$, and $d^{\frac{1}{2}}||v^{-3}b||_{L^{\infty}}<2\sqrt{\mu\sigma}$,
where $||v^{-3}b||_{L^{\infty}}:=\max_{i}||v^{-3}b_{i}||_{L^{\infty}}$. 
\item \label{enu:2} $v\in\mathcal{C}_{b}^{1}(\Omega)$ satisfies (\ref{eq:=000020w'<w/x}),
(\ref{eq:dimension}) with $p=2$ (and hence we should have $d\geq3$),
$v^{-4}b_{i}\in L^{\infty}(\Omega)$ for all $i=1,...,d$, $cv^{-2}\geq\sigma>0$
a.e. in $\Omega$, and 
\begin{equation}
0<\mu-C_{\Omega}d^{\frac{1}{2}}||v^{-4}b||_{L^{\infty}}\label{eq:1.17-4}
\end{equation}
\item \label{enu:3} $v^{-3}b_{i}\in L^{\infty}(\Omega)\cap\mathcal{C}^{1}(\Omega)$
for all $i=1,...,d$, and one of the following assertions holds 
\begin{enumerate}[label=(\alph*)]
\item \label{enu:4-1} $\text{div}(b)\leq0$ a.e. in $\Omega$ and $cv^{-2}\geq\sigma>0$
a.e. in $\Omega$
\item \label{enu:c=00005Cleq=0000200} $v^{-2}\text{div}(b)\in L^{\infty}(\Omega)$
and $(c-\frac{1}{2}\text{div}(b))v^{-2}\geq\sigma>0$ a.e. in $\Omega$. 
\end{enumerate}
\end{enumerate}
Then, the boundary-value problem (\ref{eq:First=000020PDE}) has a
unique weak solution $f\in X_{|\cdot|,v,0}^{1}(\Omega)$. In addition,
we have the estimate 

\begin{equation}
\exists\gamma>0:\,||f||_{X_{|\cdot|,v}^{1}}\leq\frac{1}{\gamma}|h|\label{eq:elliptic=000020regularity-1}
\end{equation}
where $|h|$ is the norm operator of $h$.
\end{thm}

\begin{proof}
We check that the bilinear form $B_{v}$ defined in (\ref{eq:Bilnear=000020form})
satisfies the assumptions of the Lax-Milgram theorem in the Hilbert
space $X_{|\cdot|,v,0}^{1}(\Omega)$. Note that in any cases, we have
$b_{i}v^{-3}\in L^{\infty}(\Omega)$. Hence, for any $f$, $g\in X_{|\cdot|,v,0}^{1}(\Omega)$
we have 
\begin{multline*}
\left|B_{v}(f,g)\right|\leq C\sum_{i,j=1}^{d}||v^{2}D_{v}^{i}f||_{L^{2}}||v^{2}D_{v}^{j}g||_{L^{2}}+C\sum_{i=1}^{d}||v^{2}D_{v}^{i}f||_{L^{2}}||vg||_{L^{2}}\\
+C||vf||_{L^{2}}||vg||_{L^{2}}\leq C||f||_{X_{|\cdot|,v}^{1}}||g||_{X_{|\cdot|,v}^{1}}
\end{multline*}
for some non-negative constant $C$, which proves the continuity of
the bilinear form $B_{v}$. We prove now that the bilinear form $B_{v}$
is coercive, that is, 
\begin{equation}
\exists\gamma>0\,\forall f\in X_{|\cdot|,v,0}^{1}(\Omega):\,\gamma||f||_{X_{|\cdot|,v}^{1}}^{2}\leq B_{v}(f,f).\label{eq:coercivity}
\end{equation}
Fix $f\in X_{|\cdot|,v,0}^{1}$. Under the assumptions on $(a_{ij})$
we have 
\[
\mu||v^{2}D_{v}f||_{L^{2}}^{2}+\int_{\Omega}cf^{2}\,\diff x+\int_{\Omega}f(b\cdot D_{v})f\,\diff x\leq B_{v}(f,f).
\]
Suppose that assertion \ref{enu:=0000201} holds. Then by Hölder's
inequality, we have 
\begin{equation}
\mu||v^{2}D_{v}f||_{L^{2}}^{2}+\sigma||vf||_{L^{2}}^{2}-d^{\frac{1}{2}}||v^{-3}b||_{L^{\infty}}||v^{2}D_{v}f||_{L^{2}}||vf||_{L^{2}}\leq B_{v}(f,f)\label{eq:coercive}
\end{equation}
Thanks to the assumptions of the assertion \ref{enu:=0000201}, it
is possible to choose $\gamma$ sufficiently small so that $0<\gamma\leq\min(\mu,\sigma)$
and $d^{\frac{1}{2}}||b||_{L_{v^{-3}}^{\infty}}<2\sqrt{(\mu-\gamma)(\sigma-\gamma)}$.
Thus, it follows that left hand side of (\ref{eq:coercive}) is bounded
from bellow by $\gamma(||v^{2}D_{v}f||_{L^{2}}^{2}+||vf||_{L^{2}}^{2})$
which gives (\ref{eq:coercivity}). 

In case assertion \ref{enu:2} holds, we use Hölder's inequality to
obtain 
\[
\mu||v^{2}D_{v}f||_{L^{2}}^{2}+\sigma||vf||_{L^{2}}^{2}-d^{\frac{1}{2}}||v^{-4}b||_{L^{\infty}}||v^{2}D_{v}f||_{L^{2}}||v^{2}f||_{L^{2}}\leq B_{v}(f,f)
\]
By Cor.~\ref{cor:Poincar=0000E9=000020inequality=000020}, we have
\[
(\mu-C_{\Omega}d^{\frac{1}{2}}||v^{-4}b||_{L^{\infty}})||v^{2}D_{v}f||_{L^{2}}^{2}+\sigma||vf||_{L^{2}}^{2}\leq B_{v}(f,f).
\]
Therefore, (\ref{eq:coercivity}) holds with the constant $\gamma:=\min((\mu-C_{\Omega}d^{\frac{1}{2}}||v^{-4}b||_{L^{\infty}}),\sigma)$
which is non-negative thanks to (\ref{eq:1.17-4}). 

Assume now that \ref{enu:3} holds. By (\ref{eq:ellipticity=000020condition=000020w})
we have 
\begin{equation}
\mu||v^{2}D_{v}f||_{L^{2}}^{2}+\int_{\Omega}f(b\cdot D_{v})f\,\diff x+\int_{\Omega}cf^{2}\,\diff x\leq B_{v}(f,f).\label{eq:coercive-1}
\end{equation}
Let $(\phi_{n})$ be a sequence of $\mathcal{D}(\Omega)$ that converges
to $f$ in $X_{|\cdot|,v}^{1}(\Omega)$. Since $v^{-3}b_{i}\in L^{\infty}(\Omega)$
we have 
\begin{multline*}
\left|\int_{\Omega}b_{i}(fD_{v}^{i}f-\phi_{n}\partial_{i}\phi_{n})\,\diff x\right|\leq\\
||b_{i}v^{-3}||_{L^{\infty}}||fv||_{L^{2}}||v^{2}(D_{v}^{i}f-\partial_{i}\phi_{n})||_{L^{2}}+||b_{i}v^{-3}||_{L^{\infty}}||v(f-\phi_{n})||_{L^{2}}||v^{2}\partial_{i}\phi_{n}||_{L^{2}}.
\end{multline*}
Thus,
\begin{equation}
\int_{\Omega}f(b\cdot D_{v})f\,\diff x=\lim_{n\to\infty}\int_{\Omega}\phi_{n}(b\cdot\nabla)\phi_{n}\,\diff x=-\frac{1}{2}\lim_{n\to\infty}\int_{\Omega}\text{div}(b)\phi_{n}^{2}\,\diff x.\label{eq:IBP}
\end{equation}
Hence, in case \ref{enu:4-1} of \ref{enu:3} holds, the right hand
side of the latter equality is positive. Thus, the left hand side
of (\ref{eq:coercive-1}) is bounded from bellow by $\gamma||f||_{X_{|\cdot|,v}^{1}}^{2}$
where $\gamma=\min(\mu,\sigma)>0$, and hence (\ref{eq:coercivity})
is proved, and in case \ref{enu:c=00005Cleq=0000200} of \ref{enu:3}
holds, the right hand side of (\ref{eq:IBP}) is equal to $-\frac{1}{2}\int_{\Omega}\text{div}(b)f^{2}\,\diff x$.
Replacing this in (\ref{eq:coercive-1}) we get 
\[
\gamma||f||_{X_{|\cdot|,v}^{1}}^{2}\leq\mu||v^{2}D_{v}f||_{L^{2}}^{2}+\sigma||vf||_{L^{2}}^{2}\leq B_{v}(f,f)
\]
where $\gamma:=\min(\mu,\sigma)>0$. Thus, (\ref{eq:coercivity})
holds. 

Therefore, if one of the assertions \ref{enu:=0000201}-\ref{enu:3}
holds, the bilinear form $B_{v}$ is coercive. On the other hand,
$h$ is a continuous functional on $X_{|\cdot|,v,0}^{1}$. We conclude
then by the Lax-Milgram theorem. 

For the estimate (\ref{eq:elliptic=000020regularity-1}), it follows
easily from (\ref{eq:weak=000020formulation}), (\ref{eq:coercivity}),
and the continuity of $h$. Indeed, if we denote by $f$ the unique
weak solution of (\ref{eq:First=000020PDE}), then
\[
\gamma||f||_{X_{|\cdot|,v}^{1}}^{2}\leq B_{v}(f,f)=h(f)\leq|h|||f||_{X_{|\cdot|,v}^{1}}.
\]
\end{proof}
\begin{example}
Let $v\in\mathcal{C}^{1,*}(\Omega)$, and let $(\tilde{a}_{ij})$,
$(\tilde{b})=(\tilde{b}_{i})$, $\tilde{c}$ be such that 
\begin{enumerate}
\item $\forall i,j=1,...,d:\,\tilde{a}_{ij},\,\tilde{b}_{i},\,\tilde{c}\in L^{\infty}(\Omega)$
\item $\exists\mu>0$ such that $\sum_{i,j}\tilde{a}_{ij}\xi_{i}\xi_{j}\geq\mu|\xi|^{2}$
a.e. in $\Omega$ and for all $\xi\in\R^{d}$
\item $\exists\sigma>0$ such that $\tilde{c}\geq\sigma$ a.e. in $\Omega$,
and $d^{\frac{1}{2}}||\tilde{b}||_{L^{\infty}(\Omega)}<2\sqrt{\mu\sigma}$.
\end{enumerate}
Set 
\[
\forall i,j=1,...,d:\,a_{ij}=v^{4}\tilde{a}_{ij},\,b_{i}=v^{\text{3}}\tilde{b}_{i},\,c=v^{\text{2}}\tilde{c}.
\]
Then, the functions $(a_{ij})$ satisfy (\ref{eq:ellipticity=000020condition=000020w}),
and the functions $(b_{i})$, $c$ satisfy \ref{enu:=0000201} of
Thm.~\ref{thm:second=000020application=000020}.
\end{example}

In the next proposition we give sufficient conditions under which
the solution is not locally integrable. 
\begin{prop}
\label{prop:non-loc=000020int=000020sol}Let $\Omega\subseteq\R^{d}$
be a bounded open set, $v\in\mathcal{C}^{2,*}(\Omega)$ be a map satisfying
(\ref{eq:2.9}). Let $f\in X_{|\cdot|,v,0}^{1}(\Omega)$ be such that
(\ref{eq:weak=000020formulation}) holds where 
\begin{enumerate}
\item \label{enu:=000020non-loc=000020solution=0000201}$\forall i,j=1,...,d:\,a_{ij}\in\mathcal{C}^{1}(\Omega)\cap L_{v^{-4}}^{\infty}(\Omega)$
and $\forall\alpha\in\pi_{1}^{0}:\,\partial^{\alpha}a_{ij}\in L_{v^{-3},\mathrm{loc}}^{\infty}(\Omega)$ 
\item \label{enu:=000020non-loc=000020solution=0000202}$\forall i=1,...,d:\,b_{i}\in\mathcal{C}^{1}(\Omega)\cap L_{v^{-3}}^{\infty}(\Omega)$
and $\forall\alpha\in\pi_{1}^{0}:\,\partial^{\alpha}b_{i}\in L_{v^{-2},\mathrm{loc}}^{\infty}(\Omega)$
\item \label{enu:non-local=000020solution=0000204}$c\in L_{v^{-2}}^{\infty}(\Omega)$ 
\item \label{enu:=000020non-loc=000020solution=0000203} the map $h$ is
defined by $h(g):=\int_{\Omega}kg\diff x$, where $k\in L_{v^{-1}}^{2}(\Omega)$,
$k\geq0$ $\mathrm{a.e.}$ in $\Omega$ and $kv^{-2}\not\in L_{\mathrm{loc}}^{1}(\Omega)$ 
\end{enumerate}
Then, $f\not\in L_{\mathrm{loc}}^{1}(\Omega)$. 
\end{prop}

\begin{proof}
Use (\ref{eq:weak=000020formulation}) with $g=\phi\chi_{n}v^{-2}\in\mathcal{C}_{c}^{2}(\Omega)$
(because supp$(\chi_{n})\subseteq M_{2n}^{c}$) where $\phi\in\mathcal{D}(\Omega)$
and $(\chi_{n})$ is the sequence of $\mathcal{C}^{2}(\Omega)$ functions
constructed in the proof of Prop.~\ref{prop:comparison=000020with=000020weak=000020derivative=000020}
we get 
\begin{multline}
\sum_{i,j=1}^{d}\int_{\Omega}a_{ij}D_{v}^{i}f\partial_{j}(\phi\chi_{n}v^{-2})\,\diff x+\sum_{i=1}^{d}\int_{\Omega}b_{i}\phi\chi_{n}v^{-2}D_{v}^{i}f\,\diff x\\
+\int_{\Omega}cf\phi\chi_{n}v^{-2}\,\diff x=\int_{\Omega}k\phi\chi_{n}v^{-2}\,\diff x\label{eq:=000020weak=000020f=0000202}
\end{multline}
Using the weak derivative in the sense of $L_{v,\mathrm{loc}}^{1}(\Omega)$
with $a_{ij}v^{-2}\partial_{j}(\phi\chi_{n}v^{-2})\in\mathcal{C}_{c}^{1}(\Omega)$
we get 
\begin{equation}
\int_{\Omega}a_{ij}D_{v}^{i}f\partial_{j}(\phi\chi_{n}v^{-2})\,\diff x=-\int_{\Omega}f\partial_{i}(a_{ij}\partial_{j}(\phi\chi_{n}v^{-2}))\,\diff x\label{eq:2.27}
\end{equation}
Similarly, using the weak derivative in the sense of $L_{v,\mathrm{loc}}^{1}(\Omega)$
with the test function $b_{i}\phi\chi_{n}v^{-4}\in\mathcal{C}_{c}^{1}(\Omega)$
we get 
\begin{equation}
\int_{\Omega}b_{i}\phi\chi_{n}v^{-2}D_{v}^{i}f\,\diff x=-\int_{\Omega}f\partial_{i}(b_{i}\phi\chi_{n}v^{-2})\,\diff x\label{eq:2.28}
\end{equation}
Using (\ref{eq:2.27}) and (\ref{eq:2.28}) in (\ref{eq:=000020weak=000020f=0000202})
we get 
\begin{multline}
-\sum_{i,j=1}^{d}\int_{\Omega}f\partial_{i}(a_{ij}\partial_{j}(\phi\chi_{n}v^{-2}))\,\diff x-\sum_{i=1}^{d}\int_{\Omega}f\partial_{i}(b_{i}\phi\chi_{n}v^{-2})\,\diff x\\
+\int_{\Omega}cf\phi\chi_{n}v^{-2}\,\diff x=\int_{\Omega}k\phi\chi_{n}v^{-2}\,\diff x.\label{eq:non-loc=000020sol=0000204}
\end{multline}
A simple calculation gives 
\begin{multline}
\partial_{i}(a_{ij}\partial_{j}(\phi\chi_{n}v^{-2}))=\chi_{n}(\partial_{i}a_{ij}\partial_{j}(v^{-2}\phi)+a_{ij}\partial_{ij}(v^{-2}\phi))\\
+\partial_{j}\chi_{n}(v^{-2}\phi\partial_{i}a_{ij}+a_{ij}\partial_{i}(v^{-2}\phi))+\partial_{i}\chi_{n}(a_{ij}\partial_{j}(v^{-2}\phi))+a_{ij}v^{-2}\phi\partial_{ij}\chi_{n}\label{eq:=000020non-loc=000020sol=0000201}
\end{multline}
Developing the expression $\partial_{i}a_{ij}\partial_{j}(v^{-2}\phi)+a_{ij}\partial_{ij}(v^{-2}\phi)$
we get 
\begin{multline*}
\partial_{i}a_{ij}(-2(\partial_{j}v)v^{-3}\phi+v^{-2}\partial_{j}\phi)+\\
a_{ij}(v^{-2}\partial_{ij}\phi-2(\partial_{j}v)v^{-3}\partial_{i}\phi-2(\partial_{i}v)v^{-3}\partial_{j}\phi-2(\partial_{ij}v)v^{-3}\phi+6(\partial_{i}v)(\partial_{j}v)v^{-4}\phi)
\end{multline*}
By assumption \ref{enu:=000020non-loc=000020solution=0000201}
\begin{equation}
\exists C>0\,\forall n\in\N:\,|\chi_{n}(\partial_{i}a_{ij}\partial_{j}(v^{-2}\phi)+a_{ij}\partial_{ij}(v^{-2}\phi))|\leq C1_{V}\label{eq:non-loc=000020sol=0000202}
\end{equation}
where $1_{V}$ is the characteristic function of the set $V:=\mathrm{supp}(\phi)$.
Note now that, for all $\alpha\in\pi_{2}^{0}$ we have $\partial^{\alpha}\chi_{n}\equiv0$
on $M_{n}^{c}$ (because by $\chi_{n}\equiv1$ on $M_{n}^{c}$), and
that $|\partial^{\alpha}\chi_{n}|\leq Cn^{|\alpha|}$ on $M_{n}$
for all $n$ sufficiently large, where $C>0$ is independent of $n$.
This is because $M_{n}\Subset\Omega$ (by assumption (\ref{eq:2.9}))
(and hence $\partial^{\alpha}v$ is bounded on $M_{n}$), and because
$|\eta_{n}^{(k)}|\leq Cn^{|k|}$ where the constant $C>0$ depends
on $k$ but not on $n$. It follows that 
\begin{equation}
\exists C>0\,\exists n_{0}\in\N\,\forall n\geq n_{0}\,\forall\alpha\in\pi_{2}^{0}:\,|v^{|\alpha|}\partial^{\alpha}\chi_{n}|\leq C\label{eq:2.29}
\end{equation}
since $|v|\leq1/n$ on $M_{n}$. Thus, one can easily prove (as above),
using assumption \ref{enu:=000020non-loc=000020solution=0000201}
that 
\begin{equation}
\left|\partial_{j}\chi_{n}(v^{-2}\phi\partial_{i}a_{ij}+a_{ij}\partial_{i}(v^{-2}\phi))+\partial_{i}\chi_{n}(a_{ij}\partial_{j}(v^{-2}\phi))+a_{ij}v^{-2}\phi\partial_{ij}\chi_{n}\right|\leq C1_{V}.\label{eq:loc=000020sol=0000203}
\end{equation}
for all $n\geq n_{0}$. Assume by contradiction that $f\in L_{\mathrm{loc}}^{1}(\Omega)$.
Then, by (\ref{eq:=000020non-loc=000020sol=0000201}), (\ref{eq:non-loc=000020sol=0000202})
and (\ref{eq:loc=000020sol=0000203}) we have 
\begin{equation}
\forall n\geq n_{0}:\,\left|\sum_{i,j=1}^{d}\int_{\Omega}f(\partial_{i}(a_{ij}\partial_{j}(\phi\chi_{n}v^{-2})))\,\diff x\right|\leq C||f||_{L^{1}(V)}.\label{eq:=000020non=000020loc=000020sol=0000205}
\end{equation}
By \ref{enu:=000020non-loc=000020solution=0000202} and \ref{enu:non-local=000020solution=0000204},
and using similar arguments as above one can prove that 
\begin{equation}
\forall n\geq n_{0}:\,\left|\sum_{i=1}^{d}\int_{\Omega}f\partial_{i}(b_{i}\phi\chi_{n}v^{-2})+fc\phi\chi_{n}v^{-2}\,\diff x\right|\leq C||f||_{L^{1}(V)}\label{eq:=000020non-loc=000020sol=0000206}
\end{equation}
We deduce from (\ref{eq:non-loc=000020sol=0000204}), (\ref{eq:=000020non=000020loc=000020sol=0000205}),
(\ref{eq:=000020non-loc=000020sol=0000206}) that 
\[
\exists C>0\,\forall n\geq n_{0}:\,\int_{\Omega}k\phi\chi_{n}v^{-2}\,\diff x\leq C||f||_{L^{1}(V)}
\]
Let now $K\Subset\Omega$ be such that $kv^{-2}\not\in L^{1}(K)$
and choose $\phi$ so that $\phi\geq0$ on $\Omega$ and $\phi\equiv1$
on $K$. It follows that $\int_{K}k\chi_{n}v^{-2}\,\diff x\leq C$
for all $n\geq n_{0}$, which leads to a contradiction when choosing
$n$ sufficiently large since 
\[
\lim_{n\to\infty}\int_{K}k\chi_{n}v^{-2}\,\diff x=\int_{K}kv^{-2}\,\diff x=\infty
\]
by the monotone convergence theorem. 
\end{proof}
\begin{rem}
\label{rem:non-loc=000020integrable=000020solution=000020}If 
\[
\forall i,j=1,...,d:\,a_{ij}=\tilde{a}_{ij}v^{4},\,b_{i}=\tilde{b}_{i}v^{3}
\]
where 
\[
\forall i,j=1,...,d:\,\tilde{a}_{ij}\in\mathcal{C}^{1}(\Omega),\,\tilde{b}_{i}\in\mathcal{C}^{1}(\Omega),
\]
then assumptions \ref{enu:=000020non-loc=000020solution=0000201},
\ref{enu:=000020non-loc=000020solution=0000202} of the latter proposition
hold. 
\end{rem}

\begin{example}
\label{exa:=000020Non=000020localy=000020integrabl=000020solution=000020}For
any integer $m\geq1$, and for any $\beta\in(d/2-2m,d/2)$, the boundary-value
problem 
\[
\begin{cases}
-\sum_{i=1}^{d}D_{i}(|x|^{8m}D_{i}f)+|x|^{4m}f=|x|^{-\beta+2m} & \text{in }\{x\in\R^{d}\mid|x|<1\}\\
f=0 & \text{on }\{x\in\R^{d}\mid|x|=1\}
\end{cases}
\]
has a unique non-locally integrable solution $f\in X_{|\cdot|,|x|^{2m},0}^{1}(\Omega)$
thanks to Thm.~\ref{thm:second=000020application=000020} and to
Prop.~\ref{prop:non-loc=000020int=000020sol}. 
\end{example}

\end{document}